\documentclass[11pt,a4]{article}

\usepackage{cite}
\usepackage{amsmath,amssymb,amsfonts,amsthm}
\usepackage[pdftex]{graphicx}
\usepackage{algorithm,algorithmic}
\usepackage{hyperref}

\usepackage{subcaption}
\usepackage{bm}   
\usepackage{mathrsfs}

\newtheorem{lemma}{Lemma}
\newtheorem{definition}{Definition}
\newtheorem{theorem}{Theorem}
\newtheorem{remark}{Remark}

\newtheorem{proposition}{Proposition}
\newtheorem{problem}{Problem}

\begin{document}

\title{Participation Factors for Nonlinear Autonomous Dynamical Systems in the Koopman Operator Framework\footnote{This work was partially supported by JST PRESTO Grant Number JPMJPR1926, JSPS KAKENHI Grant Number 23H01434, JST Moonshot R\&D Grant Number JPMJMS2284, and NSF Grant Number 2328241.}~\footnote{This work has been submitted to the IEEE for possible publication. Copyright may be transferred without notice, after which this version may no longer be accessible.}
}
\author{Kenji Takamichi\footnote{He was with the Department of Electrical and Information Systems, Osaka Prefecture University, Sakai, 599-8531 Japan. 
Presently, he works in the industry (e-mail: ktakamichi.public@gmail.com).}, 
Yoshihiko Susuki\footnote{He is with the Department of Electrical Engineering, Kyoto University, Kyoto, 615-8510 Japan (e-mail: susuki.yoshihiko.5c@kyoto-u.ac.jp).}, and
Marcos Netto\footnote{He is with the Helen and John C. Hartmann Department of Electrical and Computer Engineering, New Jersey Institute of Technology, Newark, NJ 07102 USA (e-mail: marcos.netto@njit.edu).}
}
\date{}
\maketitle

\begin{abstract}
We devise a novel formulation and propose the concept of modal participation factors to nonlinear dynamical systems. The original definition of modal participation factors (or simply participation factors) provides a simple yet effective metric. It finds use in theory and practice, quantifying the interplay between states and modes of oscillation in a linear time-invariant (LTI) system. In this paper, with the Koopman operator framework, we present the results of participation factors for nonlinear dynamical systems with an asymptotically stable equilibrium point or limit cycle. We show that participation factors are defined for the entire domain of attraction, beyond the vicinity of an attractor, where the original definition of participation factors for LTI systems is a special case. Finally, we develop a numerical method to estimate participation factors using time series data from the underlying nonlinear dynamical system. The numerical method can be implemented by leveraging a well-established numerical scheme in the Koopman operator framework called dynamic mode decomposition. 
\end{abstract}


\section{Introduction}
\label{sec:introduction} 

The Participation Factor (PF) is a metric that quantifies the mutual impacts between states and modes in a dynamical system, originally introduced as a component in the selective modal analysis for Linear Time-Invariant (LTI) systems \cite{PF1_1982, PF2_1982}. 
Because of its practical value in the power industry, e.g., model reduction \cite{Chow2013} or the placement of controllers \cite{Tzounas2020}, this metric has garnered the interest of many researchers over the years. 
There is a fundamental difference between the measure of participation of modes in states (mode-in-state PF) and that of participation of states in modes (state-in-mode PF) \cite{PF1_1982, Hashlamoun2009}. 
The mode-in-state PF appears as the coefficient in the dynamics of the state variable ($x_k$ defined later), whereas the state-in-mode PF appears as the coefficient in the dynamics of the modal variable ($z_j$). 
The two concepts of PF are of technological importance when the state and modal variables have individual physical meanings such as in the power grid systems. 
Note that Generalized Participation (GP) \cite{Pagola1989} extends the notion of mode-in-state PF. 
Indeed, to complement and improve the original definition of PFs, many studies have proposed alternative interpretations \cite{Garofalo2002} and definitions \cite{Tzounas2020, Abed2000, Hashlamoun2009, Song2019, Dassios2020, Iskakov2021}. 
Irrespective of a range of points of view, PFs are well-accepted indices to analyze the dynamic performance of LTI systems.

Generalizing the concept of PFs for nonlinear systems is a long-standing problem in systems theory. 
One motivation stems from power systems, where a nonlinear generalization has been required to accurately evaluate the grid's dynamic performance under a highly stressed condition: see, e.g., \cite{Sanchez-Gasca2005}.
A successful approach is to consider the higher-order terms of the Taylor series expansion of a target nonlinear system in a neighborhood of an Equilibrium Point (EP) \cite{Sanchez-Gasca2005, Hamzi2014, Tian2018, Hamzi2020}. 
This class of methods includes the normal forms for vector fields or Poincar\'{e} linearization and makes it possible to extend existing definitions of PFs. 
However, the generalized PFs in \cite{Sanchez-Gasca2005, Hamzi2014, Tian2018, Hamzi2020} are still defined \emph{locally} in the state space because of the Taylor series expansion around an EP. 
These methods are, therefore, restricted to the neighborhood of an EP. 
Indeed, they assume the existence of an EP and do not apply to nonlinear systems with global attractors, e.g., systems with an asymptotically stable Limit Cycle (LC). 
Extending the concept of PFs \emph{globally} in the basin of attraction is an open problem in modern nonlinear systems theory. 

\begin{table*}[t] 
\centering
\caption{Participation factors for linear systems \cite{PF1_1982, PF2_1982} and for nonlinear systems developed in this paper. }
\label{Table_Intro}
\footnotesize
\begin{tabular}{ccc}
\hline \hline 
& Linear systems \cite{PF1_1982, PF2_1982} & Nonlinear systems\\
\hline \noalign{\vskip 0.5mm}
System model & $\dot{\bf x} = {\bf A}{\bf x},\quad{\bf x}\in\mathbb{R}^n$ &  $\dot{\bf x} = {\bf F}({\bf x}),\quad{\bf x}\in\mathbb{R}^n$  \\ \hline 
Modal variable $z_j$  &  ${\bf u}_j^\top {\bf x}$  &  $\phi_j({\bf x})$ \rule[0mm]{0mm}{3.5mm}  \\
& (left eigenvector ${\bf u}_j$ of $\bf A$) & (eigenfunction $\phi_j({\bf x})$ of the Koopman operator ${\cal U}^t$) \\ \noalign{\vskip 0.5mm}
\hline \noalign{\vskip 1mm}
Participation factors & $u_{{jk}}
v_{{jk}}
$  & $\displaystyle \frac{\partial \phi_j}{\partial x_k}({\bf x}) V_{{jk}}$  \\ 
of state $x_k$ and mode $z_j$ & (right eigenvector ${\bf v}_j$ of ${\bf A}$) & (Koopman mode ${\bf V}_j$ for ${\cal U}^t$, corresponding to \\
& & the inner product of $\phi_j({\bf x})$ and $f({\bf x})=x_k$) \\ \noalign{\vskip 0.5mm}
\hline
Domain & entire state space & whole domain of attraction of a {stable} \\
& &  Equilibrium Point (EP) or Limit Cycle (LC) \\
\hline \hline
\end{tabular}
\end{table*}

This paper introduces a new formulation of PFs and GP that is valid globally in the basin of attraction of nonlinear systems. 
The proposed method is derived from the Koopman operator framework of nonlinear dynamical systems: see, e.g., \cite{Mezic2013, KoopmanBook, Otto2021, Bevanda2021, Brunton2022}, and Section~\ref{sec:KOtheory} of this paper. 
The Koopman operator is a linear, infinite-dimensional operator defined for a broad class of nonlinear dynamical systems.  
Even if the governing dynamics of an underlying nonlinear system are finite-dimensional, the Koopman operator is infinite-dimensional and does not rely on linearization; indeed, it captures the complete information of the nonlinear system. 
Of particular interest here, the point spectrum (eigenvalues) of the Koopman operator provides a complete spectral characterization of nonlinear systems with asymptotically stable EPs and LCs (see \cite{Mezic2013, Mauroy2013, Mezic2020}), which is valid in the whole domain of attraction. 
This paper relies on this fact to define the concept of PFs and GP for nonlinear systems.
Also, we introduce a new concept: \emph{state-in-mode} GP, which arises when we target a nonlinear system. 
The key ideas of this paper on PFs are summarized in Table~\ref{Table_Intro}, whose details will be explained later.

The main contributions of this paper are as follows. 
First, in Section~\ref{sec:LS}, we provide a novel interpretation of the PFs in \cite{PF1_1982, PF2_1982} and GP in \cite{Pagola1989} in terms of their variational dynamics. 
The idea of variational dynamics has been widely utilized in the classical theory of nonlinear oscillations \cite{Parker1989}, ergodic theory of chaos including the Lyapunov exponent \cite{Pikovsky2016}, the theory of phase-amplitude reduction \cite{Shirasaka2017}, and nonlinear systems theory, e.g., contraction theory \cite{Tsukamoto2021}.  
The variational idea is crucial to establish the concept of PFs and GPs for nonlinear systems in this paper. 
By virtue of this reinterpretation, we show that the classical concept of PFs and GPs defined for LTI systems is a special case of the concept proposed in this paper for nonlinear systems.

Second, in Section~\ref{sec:NLS}, following the spectral characterization of the Koopman operator \cite{Mezic2013, Mauroy2013, Mezic2020}, we provide novel definitions of the PFs and GPs. 
This becomes possible by combining the variational interpretation above with the spectral characterization, technically, eigenfunctions of the Koopman operator, referred to as the Koopman eigenfunctions, and Koopman Mode Decomposition (KMD) \cite{Mezic2013, Mezic2005, Rowley2009}. 
We thereby show that the novel PFs and GPs are defined globally in the whole domain of attraction of a stable EP or LC. 
This unlocks a range of applications for PFs and GPs in systems and control. 
For example, it becomes possible to evaluate the PFs for a state far from an attractor and the time evolution of PFs along a state's trajectory, which is necessary for transient stability and stabilization of power systems \cite{Iskakov2021, Takamichi2022, Zheng2022, Kotb2017}. 

Third, in Section~\ref{sec:NA}, we develop a numerical method to estimate the PFs proposed in this paper. 
The proposed PFs in Table~\ref{Table_Intro} require the estimation of Koopman eigenfunctions $\phi_j({\bf x})$, which is still a challenging issue in the Koopman operator framework, especially, for large-scale problems such as power systems \cite{Netto2021}.  
The method in this paper adopts the so-called variational equation \cite{Parker1989} and hence makes it possible to numerically estimate the proposed PFs directly from time series data \emph{without} estimating the Koopman eigenfunctions.
It can be implemented with a well-established numerical scheme in the Koopman operator framework--Dynamic Mode Decomposition (DMD) \cite{Kutz2016}. 
Ultimately, the proposed method entails a simple linear regression problem. 
In contrast, existing approaches based on local approximation may suffer from a highly nonlinear numerical problem \cite{Sanchez-Gasca2005, Netto2019}.

We now distinguish this paper from the authors' previous works \cite{Takamichi2022, Netto2019}.
In \cite{Netto2019}, the KMD was utilized for the first time to formulate the PFs for nonlinear systems.
However, the definitions of PFs in \cite{Netto2019} are clearly different from those defined in this paper. 
A crucial distinction is the state's dependency: in \cite{Netto2019}, by taking an expectation over initial states like \cite{Hashlamoun2009},  we obtained the \emph{state-independent} expression for the PFs while in this paper, by the variational approach, we theoretically formulate the \emph{state-dependent} PFs.
In \cite{Takamichi2022}, for nonlinear systems, we derived the state-dependent PF in the sense of mode-in-state, using the KMD. 
There, we explored the notion of an infinitesimal change to the state's evolution.
This paper is the substantially extended version of \cite{Takamichi2022}. 
This paper explicitly formulates the PFs and GPs for linear systems from the perspective of variational dynamics and newly develops the state-in-mode PF and generalized GPs which have yet to be reported in \cite{Takamichi2022}. 
Consequently, this paper achieves the unified definitions of PFs and GPs for nonlinear systems.
Further, we newly devise a numerical method to estimate the PFs without requiring the estimation of Koopman eigenfunctions. 

The rest of this paper is organized as follows.
Section~\ref{sec:PS} contains the classical theory of PFs and GPs and formulates the problem in this paper.
Section~\ref{sec:LS} reinterprets the PFs and GPs for LTI systems from the viewpoint of variational dynamics. 
By combining the variational viewpoint with the Koopman operator framework, Section~\ref{sec:NLS} proposes PFs and GPs for nonlinear systems; this is the main result of this paper.  
In Section~\ref{sec:examples}, the proposed PFs and GPs are demonstrated with simple models of the nonlinear systems.
Section~\ref{sec:NA} develops a numerical method to estimate the proposed PFs and demonstrates it in the same examples in Section~\ref{sec:examples}. 
Section~\ref{sec:outro} concludes this paper with a summary and future directions. 

\emph{Notation---}The sets of all real, complex, integer, and natural numbers are denoted by $\mathbb{R}$, $\mathbb{C}$, $\mathbb{Z}$, and $\mathbb{N}$, respectively. 
The union of $\mathbb{N}$ and $\{0\}$ is denoted by $\mathbb{N}_0$. 
The imaginary unit is denoted by ${\rm i}:=\sqrt{-1}$.
The bold font, for example, $\bf x$ and $\bf A$, is used for column vectors and matrices.  
The transpose operation of vectors is denoted by $\top$. 
For vector $\bf x$ and indexed vector ${\bf x}_j$, its $k$-th element is represented as $x_k$ and $x_{{jk}}$. 
For $\bf x$ and $x_k$, their initial values at $t=0$, are represented as ${\bf x}^0$ and $x_k^0$. 
A function $f$ with $r\,(\geq 1)$ continuous derivatives is said to be of class $\mathscr{C}^r$.

\section{Problem Statement}
\label{sec:PS}

Consider an LTI system on the $n$-dimensional Euclidean space $\mathbb{R}^n$, described by
\begin{equation}
\frac{{\rm d}{\bf x}}{{\rm d}t} = \dot{\bf x} = {\bf A}{\bf x},
\label{LTI}
\end{equation}
where ${\bf x}\in\mathbb{R}^n$ is a state vector, $t\in\mathbb{R}$ is continuous time, and ${\bf A}\in\mathbb{R}^{n\times n}$ is a system matrix. 
For simplicity of our analysis, we suppose that $\bf A$ has $n$ distinct eigenvalues, denoted by $\lambda_j$ ($j=1,\ldots,n$). 
The corresponding right and left eigenvectors of $\bf A$ are respectively denoted by ${\bf v}_j$ and ${\bf u}_j\in\mathbb{C}^n$, satisfying the orthonormal condition
\begin{equation}
{\bf u}^\top_j{\bf v}_k = \left\{\begin{array}{cl}
1, & j= k,\\
0, & j\neq k. 
\end{array}
\right.
\label{Kronecker}
\end{equation}
We now represent the solution of \eqref{LTI} starting from ${\bf x}^0$ at $t=0$ by $\bar{\bf x}(t; 0, {\bf x}^0)$ or simply $\bar{\bf x}(t; {\bf x}^0)$. 
Then, for the $k$-th element of $\bf x$, its solution is analytically represented in the so-called mode decomposition:
\begin{align}
\bar{x}_k(t; {\bf x}^0)
& = \sum_{j=1}^{n} {\rm e}^{\lambda_j t}({\bf u}_j^\top {\bf x}^0) v_{{jk}}  \label{xk_KMD_LTI} \\
& = \sum_{j=1}^{n} {\rm e}^{\lambda_j t} (u_{{jk}} v_{{jk}})x^0_k 
+ \sum_{j=1}^{n} \sum_{\substack{\ell = 1\\\ell \neq k}}^n {\rm e}^{\lambda_j t} (u_{{j\ell}}v_{{jk}})x^0_\ell,
\label{xk_MinS_LTI}
\end{align}
where $v_{jk}$ and $u_{jk}$ are the $k$-elements of the right eigenvector ${\bf v}_j$ and left eigenvector ${\bf u}_j$, respectively. 
The decomposition \eqref{xk_MinS_LTI} leads to the definitions of so-called \emph{mode-in-state} PF and GP.
\begin{definition}
\label{def_MinSPF_Linear} 
For the LTI system \eqref{LTI}, the mode-in-state participation factor \cite{PF1_1982, PF2_1982} is 
\begin{equation}
{P_{j}^{k}} := u_{{jk}}v_{{jk}}.
\label{p_MinSPF_LTI}
\end{equation}
\end{definition}
\begin{remark}
The main concern behind mode-in-state is with time evolution of the state variable $\bf x$ that is of physical importance. The mode-in-state PF \eqref{p_MinSPF_LTI} appears as the coefficient in the dynamics of the state variable $x_k$ \eqref{xk_MinS_LTI} and quantifies the relative contribution of the $j$-th mode in the state’s evolution of the $k$-th element $x_k$. 
Provided that ${\bf x}^0$ is the $k$-th unit vector ${\bf e}_k$ in $\mathbb{R}^n$, the second term on the right-hand side in \eqref{xk_MinS_LTI} vanishes, yielding the classical notion of mode-in-state PF in \cite{PF1_1982, PF2_1982}.
Assuming $t=0$ in \eqref{xk_MinS_LTI}, $\sum_{j=1}^n {P_{j}^{k}} = 1$ holds, indicating that ${P_{j}^{k}}$ quantifies the excitation of j-th mode in the $k$-th state.
\label{remark:misPF_LTI}
\end{remark}
\begin{definition}
\label{def_MinSGP_Linear} 
For the LTI system \eqref{LTI}, the mode-in-state generalized participation \cite{Pagola1989} is 
\begin{equation}
P_{j}^{k(\ell)} := u_{{j\ell}} v_{{jk}}.
\label{p_MinSGP_LTI}
\end{equation}
\end{definition}
\begin{remark}
The mode-in-state GP, $P_{j}^{k(\ell)}$, quantifies the impact of the $j$-th mode on $\bar{x}_k(t; {\bf x}^0)$ caused by the initial change in $x_\ell$.
Note that for $\ell = k$, $P_{j}^{k(k)}$ is the mode-in-state PF ${P_{j}^{k}}$.
Hence, the mode-in-state GP, $P_{j}^{k(\ell)}$, extends the notion of mode-in-state PF, $P_{j}^k$.
\label{remark:misGP_LTI}
\end{remark}

In addition to the state variable, the modal variable $z_j:={\bf u}_j^\top{\bf x}$ ($j=1,\ldots,n$) and associated time evolution $\bar{z}_j(t; {\bf x}^0)={\bf u}^\top_j\bar{\bf x}(t; {\bf x}^0)$ are represented as follows:
\begin{align}
\bar{z}_j(t; {\bf x}^0) 
&= {\bf u}_j^\top \sum_{i=1}^n {\rm e}^{\lambda_i t}({\bf u}_i^\top {\bf x}^0 ){\bf v}_i  \notag 
\\
&= {\rm e}^{\lambda_j t} {\bf u}_j^\top  {\bf v}_j (\underbrace{ {\bf u}_j^\top {\bf x}^0}_{z_j^0})  
+ \sum_{\substack{i = 1\\ i \neq j}}^n {\rm e}^{\lambda_i t} {\bf u}_j^\top {\bf v}_i (\underbrace{ {\bf u}_i^\top {\bf x}^0}_{z_i^0})  \label{zj_pre}  \\
&= {\rm e}^{\lambda_j t} \left(\sum_{k=1}^n u_{{jk}} v_{{jk}} \right) z_j^0{,}  \label{zj}
\end{align}
where $z_j^0$ is the $j$-th modal variable at $t=0$.
The second term on the right-hand side of \eqref{zj_pre} vanishes because of the condition \eqref{Kronecker}.
The decomposition \eqref{zj_pre} and \eqref{zj} leads to the definitions of so-called \emph{state-in-mode} PF and GP.
\begin{definition}
\label{def_SinMPF_Linear} 
For the LTI system \eqref{LTI}, the state-in-mode PF \cite{PF1_1982, PF2_1982} is
\begin{equation}
{P_{j}^{k}} := u_{{jk}} v_{{jk}}.
\label{p_SinM_LTI}
\end{equation}
\end{definition}
\begin{remark}
The main concern behind state-in-mode is with time evolution of the modal variable $z_j$. The state-in-mode PF \eqref{p_SinM_LTI} appears as the coefficient in the dynamics of the modal variable $z_j$ \eqref{zj} and quantifies the relative contribution of the $k$-th state in the evolution of $j$-th modal variable $z_j$, i.e., $\bar{z}_j(t; {\bf x}^0)$. 
Assuming ${\bf x}^0={\bf v}_j$ in \eqref{xk_MinS_LTI}, a single $j$-th mode is excited, which is again the classical introduction to the state-in-mode PF given in \cite{PF1_1982, PF2_1982}. 
Assuming $t=0$ in \eqref{zj}, $\sum_{k=1}^n {P_{j}^{k}} = 1$ holds, indicating that ${P_{j}^{k}}$ quantifies the contribution of the $k$-th state in $j$-th mode. 
\label{remark:simPF_LTI}
\end{remark}
\begin{remark} 
The original expressions obtained for the mode-in-state \eqref{p_MinSPF_LTI} and state-in-mode \eqref{p_SinM_LTI} PF are identical \cite{PF1_1982, PF2_1982}, although the mode-in-state PF and state-in-mode PF are derived through different ways.
Here, we should mention that the authors of \cite{Hashlamoun2009} proposed a dichotomy between participation factors---their work unveiled an inconsistency resulting from the definition \eqref{p_SinM_LTI} of the \emph{state-in-mode} PF; see \cite{Hashlamoun2009} for more details.
\end{remark}
\begin{definition} 
\label{p_SinMGP_LTI}
In parallel to Definition~\ref{def_MinSGP_Linear}, for the LTI system \eqref{LTI} the state-in-mode GP is
\begin{equation}
{P_{i(j)}^{k}} := u_{{jk}} v_{{ik}}.
\label{p_SinMGP_LTI}
\end{equation}
\label{def_SinMGP_Linear} 
\end{definition}
\begin{remark}
The state-in-mode GP, $P_{i(j)}^{k}$, weights the impact $x_{k}$ has, through $z_{i}$, on the time evolution of $z_{j}$. 
Note that for $i=j$, ${P_{i(j)}^{k}}$ is the state-in-mode PF ${P_{j}^{k}}$.
Hence, the state-in-mode GP $P_{i(j)}^{k}$ extends the notion of state-in-mode PF ${P_{j}^{k}}$. 
\end{remark}
\begin{remark}
For the LTI system \eqref{LTI}, the concept of state-in-mode GP is merely formal and not of practical significance.
That is because, for the LTI system \eqref{LTI}, the state-in-mode GP $P_{i(j)}^{k}$ does not appear in the time evolution of $z_j(t)$ in \eqref{zj}, since the second term in \eqref{zj_pre} vanishes thanks to \eqref{Kronecker}.
However, in the later section, the state-in-mode GP will play a role in the case of nonlinear systems.
\label{remark:simGP_LTI}
\end{remark}
\begin{remark}
\label{S_Indipendent}
All the definitions of PFs and GPs for the LTI system \eqref{LTI} are state-independent.
\end{remark}

This paper addresses the long-standing problem of developing the concept of PFs and GPs for nonlinear dynamical systems. 
Consider the nonlinear system described by the ordinary differential equation 
\begin{equation}
\dot{\bf x} = {\bf F} ({\bf x}),
\label{Nonlinear_ODE}
\end{equation}
where ${\bf x} \in {\mathbb R}^n$ is a state vector and ${\bf F}: {\mathbb R}^n \rightarrow {\mathbb R}^n$ is a vector-valued nonlinear function.
We make the standing assumption that the solution
\[
\bar{\bf x}(t; 0, {\bf x}^0)={\bf S}^t({\bf x}^0), \qquad \forall t\geq 0
\]
associated with the initial condition ${\bf x}^0$ exists and is unique. 
For example, the existence of the solution for all positive time can be obtained when ${\bf F}$ is globally Lipschitz continuous on a forward invariant set. 
The one-parameter family of maps ${\bf S}^t: \mathbb{R}^n\to\mathbb{R}^n, t\geq 0$ is the \emph{semi-flow} generated by \eqref{Nonlinear_ODE}. 
A positive semi-orbit, $\{{\bf S}^t({\bf x}^0)\}_{t\geq0}$, is generally complex and, for example, may exhibit a chaotic attractor.
The problem we address in this paper is the following:
\begin{problem} 
For the nonlinear system \eqref{Nonlinear_ODE}, find formulations for PFs and GPs that are consistent with those of the LTI system \eqref{LTI}.
\end{problem}

\section{Linear Systems}
\label{sec:LS}

This section restates the classical definitions of PFs and GPs for the LTI system \eqref{LTI} from the viewpoint of variational dynamics.
This restatement leads to a natural extension of those concepts to nonlinear systems, which will be exploited in Section~\ref{sec:NLS}. 
For a given initial change ${\delta{\bf x}}\in\mathbb{R}^n$, the solution of $k$-th element of $\bf x$, starting at ${\bf x}^0+{\delta{\bf x}}$, is given by
\begin{align}
\bar{x}_k(t; {\bf x}^0 + {\delta{\bf x}}) 
&= \sum_{j=1}^{n} {\rm e}^{\lambda_j t} 
{\bf u}_j^\top ({\bf x}^0 + {\delta{\bf x}}) v_{{jk}} \notag \\
&= \bar{x}_k(t; {\bf x}^0) 
+ \underbrace{\sum_{j=1}^{n} {\rm e}^{\lambda_j t} (u_{{jk}}v_{{jk}}) \delta x_{k} 
+ \sum_{j=1}^{n} \sum_{\substack{\ell = 1\\\ell \neq k}}^{n} {\rm e}^{\lambda_j t}
(u_{{j\ell}}v_{{jk}})\delta x_{\ell}}_{{{\rm d}x}_k(t; {\delta{\bf x}})},
\label{xk_vari}
\end{align}
where ${{\rm d}x}_k(t; {\delta{\bf x}})$ is the variation between the solutions ${\bar x}_k(t; {\bf x}^0 + {\delta{\bf x}})$ and $\bar{x}_k(t; {\bf x}^0)$. 
In addition to the state, the time evolution of the $j$-th mode $z_j$ starting at ${\bf x}^0+{\delta{\bf x}}$ is represented as follows:
\begin{align}
\bar{z}_j(t; {\bf x}^0 + {\delta{\bf x}})
&= {\bf u}_j^\top \bar{\bf x}(t; {\bf x}^0 + {\delta{\bf x}}) \notag \\
&= \bar{z}_j(t; {\bf x}^0) 
+ \underbrace{{\rm e}^{\lambda_j t}
\left(\sum_{k=1}^n u_{{jk}}v_{{jk}}\right){\bf u}_j^\top{\delta{\bf x}}}_{{{\rm d}z}_j(t; {\delta{\bf x}})},
\label{zk_vari}
\end{align}
where ${{\rm d}z}_j(t; {\delta{\bf x}})$ is the variation between $\bar{z}_j(t; {\bf x}^0 + {\delta{\bf x}})$ and $\bar{z}_j(t; {\bf x}^0)$. 
Here, the initial variation
\begin{equation}
{\rm d}z_j(0,\delta{\bf x})={\bf u}_j^\top\delta{\bf x}=:\delta z_j^0,
\label{eqn:delta-z_j_LTI}
\end{equation}
corresponds to the initial change to the $j$-th mode derived by projecting $\delta{\bf x}$ onto the space spanned by ${\bf u}_j$. 
Equations~\eqref{xk_vari} and \eqref{zk_vari} lead to characterizing the variational dynamics in terms of the PFs and GPs in the following proposition (whose proof follows directly from the above derivation).
\begin{proposition}
\label{prop:Variation_LTI}
Consider the LTI system \eqref{LTI} with the PFs ${P_{j}^{k}}$ and mode-in-state GP $P_{j}^{k(\ell)}$ from Definitions~\ref{def_MinSPF_Linear} to \ref{def_SinMPF_Linear}. 
For a given initial change $\delta{\bf x}\in\mathbb{R}^n$, the variational dynamics in \eqref{xk_vari} and \eqref{zk_vari} of the state ($x_k$) and the modal ($z_j$) variables are represented as follows:
\begin{align}
{{\rm d}x}_k(t; {\delta{\bf x}}) &= 
\sum_{j=1}^{n} {\rm e}^{\lambda_j t} {P_{j}^{k}} \delta x_{k} + 
\sum_{j=1}^{n} \sum_{\substack{\ell = 1\\\ell \neq k}}^{n} {\rm e}^{\lambda_j t} \label{dx_LTI}
P_{j}^{k(\ell)}\delta x_{\ell}, \\
{{\rm d}z}_j(t; {\delta{\bf x}}) &=
{\rm e}^{\lambda_j t}
\left(\sum_{k=1}^n {P_{j}^{k}}\right){\delta z_j^0}, 
\label{dz_LTI}
\end{align}
where $\delta z_j^0$ is the initial change to the $j$-th mode given in \eqref{eqn:delta-z_j_LTI}. 
\end{proposition}
\begin{remark}
\label{remark:Physical_LTI}
The mode-in-state PF ${P_{j}^{k}}$ and GP $P_{j}^{k(\ell)}$ are interpreted as metrics that evaluate the relative contribution of the $j$-th mode in the variational dynamics of the $k$-th state variable $x_k$. 
The state-in-mode PF ${P_{j}^{k}}$ is interpreted as the metric that evaluates the relative contribution of the $k$-th state element $x_k$ in the variational dynamics of the $j$-th mode $z_j$. 
Here, the state-in-mode GP ${P_{i(j)}^{k}}$ can be defined as in \eqref{p_SinMGP_LTI} but does not appear in the variational dynamics of $z_j$ for LTI systems.
We will exploit the connection of the PFs and GPs with the variational dynamics for our analysis of nonlinear systems.
\end{remark}
\begin{remark}
\label{remark:uniform_LTI}
The variational dynamics \eqref{dx_LTI} and \eqref{dz_LTI} are independent of the initial state ${\bf x}^0$ and thus \emph{uniform} over the entire state space. 
This is intrinsic to the linear vector field, that is, a linear mapping from the state space to its tangent bundle that assigns to each state a vector in the corresponding tangent space. 
\end{remark}
\begin{remark}
\label{remark1}
As in Remarks~\ref{remark:misPF_LTI}, \ref{remark:misGP_LTI} and \ref{remark:simPF_LTI}, the classical definitions of PFs and GPs are based in the Euclidean space $\mathbb{R}^n$ (namely, using the vectors ${\bf e}_k$ and ${\bf v}^j$). 
In this sense, the classical definitions are not supported to work in linear vector fields on a non-Euclidean space (generally a differentiable manifold): for example, the rotation on the circle,
\[
\dot{\theta} = \omega, \quad \theta \in \mathbb{R}/\mathbb{Z},
\label{Torus} 
\]
where $\omega\in\mathbb{R}$ is a parameter (angular frequency of the rotation). 
However, the variational dynamics \eqref{dx_LTI} and \eqref{dz_LTI} can be introduced for the non-Euclidean case. Therefore, we show here that the traditional PFs and GPs have dynamical meaning and can work for general linear vector fields on non-Euclidean spaces. 
\end{remark}

\section{Nonlinear Systems}
\label{sec:NLS}

This section proposes novel definitions of PFs and GPs for a class of nonlinear systems by exploiting the variational formulation in Section~\ref{sec:LS}. 
The main tool comes from the Koopman operator framework summarized in the next subsection. 
In this paper, following \cite{Mauroy2013, Mezic2020, Susuki2021, Kvalheim2021}, we consider nonlinear systems with two types of non-stationary dynamics: convergent dynamics to a stable EP or LC. 
The stable EP for the nonlinear system \eqref{Nonlinear_ODE} is denoted by ${\bf x}^\ast$ with a basin of attraction $\mathbb{B}({\bf x}^\ast)\subseteq \mathbb{R}^n$. 
Following \cite{Mauroy2013}, we assume for the nonlinear system \eqref{Nonlinear_ODE} that ${\bf F}$ is analytic and that the Jacobian matrix computed at ${\bf x}^\ast$ has $n$ distinct (non-resonant\footnote{The eigenvalues $\lambda_1,\ldots,\lambda_n$ are said to be non-resonant if there exists no $(j_1,\ldots,j_n)\in\mathbb{Z}^n\setminus\{{{\bf 0}}\}$ such that $j_1\lambda_1+\cdots+j_n\lambda_n=0$.}) eigenvalues $\lambda_j$ characterized by strictly negative real parts. 
The setting of the system with a stable LC will be introduced before Theorem~\ref{Case_LC}.

\subsection{Summarized theory of Koopman operators}
\label{sec:KOtheory}

We now summarize the existing results on the Koopman operator framework that is used for our development in later sections of this paper. 
For this, we introduce a scalar-valued function, called an observable $f: \mathbb{R}^n\to\mathbb{C}$. 
Then, the definition of the Koopman operator and its spectral properties are introduced below.
\begin{definition}
Consider a (Banach) space $\mathscr{F}$ of observables $f: \mathbb{R}^n\to\mathbb{C}$. 
The family of Koopman operators $\mathcal{U}^t: \mathscr{F}\to\mathscr{F}$ associated with the family of maps ${\bf S}^t: \mathbb{R}^n\to\mathbb{R}^n, t\geq 0$, is defined through the composition 
\[
\mathcal{U}^tf=f\circ{\bf S}^t \qquad \forall f\in\mathscr{F}.
\] 
\end{definition}
\begin{definition}
A Koopman eigenvalue of $\mathcal{U}^t$ is a complex number $\lambda$ such that there exists a nonzero $\phi\in\mathscr{F}$, called a Koopman eigenfunction, such that
\[
\mathcal{U}^t\phi={\rm e}^{\lambda t}\phi \qquad \forall t\geq 0.
\]
\end{definition}
Several studies have characterized the spectral properties of the Koopman operator connected to the dynamics of the underlying nonlinear system \eqref{Nonlinear_ODE}. 
In the following, the invariance of $\mathscr F$ with respect to ${\mathcal U}^t$ is guaranteed. 
\begin{lemma}
\label{lemma:KEF_EP}
Consider the nonlinear system \eqref{Nonlinear_ODE} possessing a stable EP ${\bf x}^\ast$ with a basin of attraction $\mathbb{B}({\bf x}^\ast)$. 
Under a suitable choice of the space $\mathscr{F}$ of observables $f: \mathbb{B}({\bf x}^\ast)\to\mathbb{C}$, there exists a set of $n$ Koopman principal eigenvalues corresponding to the linearized eigenvalues {$\lambda_j$ ($j=1,\ldots,n$)} and associated Koopman eigenfunctions $\phi_j\in\mathscr{F}\setminus\{0\}$ that are smooth over $\mathbb{B}({\bf x}^\ast)$:
\[
\mathcal{U}^t\phi_j={\rm e}^{\lambda_jt}\phi_j \qquad \forall t\geq 0. \label{KE}
\] 
Furthermore, there exist Koopman eigenvalues ${j_1}\lambda_1+\cdots +{j_n}\lambda_n$ for $j_1,\ldots,j_n\in\mathbb{N}_0$ satisfying $j_1+\cdots+j_n>1$ and associated Koopman eigenfunctions defined as
\[
\phi_{\langle j_1\cdots j_n\rangle}({\bf x}):=\phi_1({\bf x})^{j_1}\cdots\phi_n({\bf x})^{j_n},
\] 
which are smooth over $\mathbb{B}({\bf x}^\ast)$:
\[
\mathcal{U}^t
\phi_{\langle j_1\cdots j_n\rangle}
={\rm e}^{({j_1}\lambda_1+\cdots +{j_n}\lambda_n)t} \phi_{\langle j_1\cdots j_n\rangle} \qquad \forall t\geq 0.
\]
\end{lemma}
\begin{proof}  
See Section~4.2.2 on page 2023 of \cite{Susuki2021}. 
\end{proof}
\begin{lemma}
\label{lemma:MV_EP}
The Koopman principal eigenfunction $\phi_j$ in Lemma~\ref{lemma:KEF_EP} provides a nonlinear generalization of the modal variable in the LTI system \eqref{LTI}. 
That is, by defining
\[
z_j:=\phi_j({\bf x}), \quad j=1,\ldots,n,
\]
its time evolution $\bar{z}_j(t; {\bf x}^0)=\phi_j({\bf S}^t({\bf x}^0))$ is represented as follows:
\[
\bar{z}_j(t; {\bf x}^0) = {\rm e}^{\lambda_j t} \bar{z}_j(0; {\bf x}^0) 
 = {\rm e}^{\lambda_jt}\phi_j({\bf x}^0).
\]
The generalization holds for the higher-order Koopman eigenfunctions, that is, 
\[
z_{\langle j_1\cdots j_n\rangle}({\bf x}) := \phi_{\langle j_1\cdots j_n\rangle}({\bf x}),
\]
and
\[
\begin{aligned}
\bar{z}_{\langle j_1\cdots j_n\rangle}(t; {\bf x}^0) 
&= {\rm e}^{({j_1}\lambda_1+\cdots +{j_n}\lambda_n)t} \bar{z}_{\langle j_1\cdots j_n\rangle}(0; {\bf x}^0) \notag\\
&= {\rm e}^{({j_1}\lambda_1+\cdots +{j_n}\lambda_n)t} \phi_{\langle j_1\cdots j_n\rangle}({\bf x}^0). 
\end{aligned}
\]
\end{lemma}
\begin{proof}
This proof is straightforward and thus omitted.
\end{proof}

The spectral property makes it possible to decompose the state dynamics of a nonlinear system, as originally invented in \cite{Mezic2005, Rowley2009}, and later called the Koopman Mode Decomposition (KMD).
\begin{lemma}
\label{lemma:KMD_EP}
Let the nonlinear system \eqref{Nonlinear_ODE} {possess} a stable EP ${\bf x}^\ast$ with a basin of attraction $\mathbb{B}({\bf x}^\ast)$. 
Assume that $\mathscr{F}$ is the modulated Segal-Bargmann space defined on $\mathbb{B}({\bf x}^\ast)$. 
Then, for all ${\bf x}^0\in\mathbb{B}({\bf x}^\ast)$ and $t\geq 0$, the flow of the system is represented as follows:
\begin{align} 
\bar{\bf x}(t; {\bf x}^0) 
&= {\bf x}^\ast+
\sum_{j=1}^{n} {\rm e}^{\lambda_jt}\phi_j ({\bf x}^0) {\bf V}_{j}
+ \sum^\infty_{\substack{{j_1},\ldots,{j_n}\in\mathbb{N}_0 \\ {j_1}+\cdots+{j_n}>1}}
{\rm e}^{({j_1}\lambda_1+\cdots +{j_n}\lambda_n)t}
\phi_{\langle j_1\cdots j_n\rangle}({\bf x}^0)
{\bf V}_{\langle j_1\cdots j_n\rangle},
\label{KMD_X} 
\end{align}
where $\lambda_j$ and $\phi_j$ are as in Lemma~\ref{lemma:KEF_EP}, and ${\bf V}_j$ and ${\bf V}_{\langle j_1\cdots j_n\rangle}$ are the vector-valued coefficients for the expansion that are uniquely determined.
\end{lemma}
\begin{proof}
See Section~4.2.2 in page~2024 of \cite{Susuki2021} for a rigorous proof of convergence. 
The modulated Segal-Bargmann space is a reproducing kernel Hilbert space, which was originally introduced in \cite{Mezic2020}: see Appendix~\ref{app:Segal} in detail. 
\end{proof}
\begin{remark}
\label{remark:KMD_EP}
The vectors ${\bf V}_j$ and ${\bf V}_{\langle j_1\cdots j_n\rangle}$ are called Koopman modes. 
The third term in \eqref{KMD_X} accounts for higher-order terms in the time evolution of $\bf x$.
The $k$-th element of ${\bf V}_j$ (or ${\bf V}_{\langle j_1\cdots j_n\rangle}$) corresponds to the inner product (equipped with $\mathscr{F}$) of the Koopman eigenfunction $\phi_j({\bf x})$ (or $\phi_{\langle j_1\cdots j_n\rangle}({\bf x})$) and {the observable} $f({\bf x})=x_k$. 
For linear systems \eqref{LTI}, higher-order terms do not appear in $\bar{\bf x}(t; {\bf x}^0)$ because the Koopman mode ${\bf V}_{\langle j_1\cdots j_n\rangle}$ is $\bf 0$ in \eqref{KMD_X}. 
\end{remark}

\subsection{Main results}

We can now introduce the {definitions} of PFs and GPs for nonlinear systems based on the Koopman eigenfunctions and KMD.
\begin{definition}
\label{def_PFGP_NL}
For the $j$-th Koopman principle eigenfunction $\phi_j$, the mode-in-state and state-in-mode Participation Factors (PFs) for the nonlinear system \eqref{Nonlinear_ODE} are defined as
\begin{equation} 
{P_{j}^{k}}({\bf x}) :=\frac{\partial \phi_j}{\partial x_k}({\bf x})V_{{jk}}, 
\quad k=1,\ldots,n.
\label{PF_NL}
\end{equation}
Also, the mode-in-state GP $P_{j}^{k(\ell)}({\bf x})$ and state-in-mode GP $P_{i(j)}^{k}({\bf x})$ for the nonlinear system  \eqref{Nonlinear_ODE}
are defined as 
\begin{align} 
\text{mode-in-state:} & \quad  P_{j}^{k(\ell)}({\bf x}) := \frac{\partial \phi_j}{\partial x_\ell}({\bf x})V_{{jk}},  
\label{GP_MinS_NL} \\
\text{state-in-mode:}& \quad P_{i(j)}^{k}({\bf x}) := \frac{\partial \phi_j}{\partial x_k}({\bf x})V_{{ik}},
\label{GP_SinM_NL}
\end{align}
where $\ell, i \in\{1,\ldots,n\}$. 
\end{definition}
\begin{definition}
\label{def_PFGP_NL_High}
For the higher-order Koopman eigenfunction  $\phi_{\langle j_1\cdots j_n\rangle}$, the Participation Factors (PFs) and Generalized Participations (GPs) for the nonlinear system \eqref{Nonlinear_ODE} are defined as
\begin{align} 
P_{\langle j_1\cdots j_n\rangle}^{k}({\bf x}) &:=
\frac{\partial \phi_{\langle j_1\cdots j_n\rangle}} {\partial x_k}({\bf x}){V}_{\langle j_1\cdots j_n\rangle k}, 
\label{PF_NL_High}\\
P_{\langle j_1\cdots j_n\rangle}^{k(\ell)}({\bf x}) &:=
\frac{\partial \phi_{\langle j_1\cdots j_n\rangle}} {\partial x_\ell}({\bf x}){V}_{\langle j_1\cdots j_n\rangle k}, 
\label{GP_MinS_NL_High} \\
P_{\langle j_1\cdots j_n\rangle (j)}^{k}({\bf x}) &:=
\frac{\partial \phi_j}  {\partial x_k}({\bf x}){V}_{\langle j_1\cdots j_n\rangle k},
\label{GP_SinM_NL_High}
\end{align}
where $j, k,\ell\in\{1,\ldots,n\}$.
\end{definition}
\begin{remark}
The definitions \eqref{PF_NL} to \eqref{GP_SinM_NL_High} of PFs and GPs are valid as long as the Koopman principal eigenfunctions $\phi_j$ are of class $\mathscr{C}^1$, and $V_{{jk}}$ (mathematically, the inner product of $\phi_j({\bf x})$ and $f({\bf x})=x_k$) is available. 
The previous condition on the Koopman eigenfunctions is guaranteed in Lemma~\ref{lemma:KEF_EP} thanks to their smoothness property. 
This is generally clarified in \cite{Kvalheim2021}. 
Furthermore, in \cite{Kvalheim2021}, the differentiability property of the Koopman eigenfunctions is clarified for a wider class of the nonlinear system \eqref{Nonlinear_ODE} such as $\mathcal{C}^r$ class beyond the analytic one assumed at the beginning of this section. 
In this sense, the defined PFs and GPs can work for a wider class of nonlinear systems.
\end{remark}

\begin{remark}\label{Rem_SinM}
The definitions \eqref{PF_NL} to \eqref{GP_SinM_NL_High} are similar.
For example, consider $P^k_j({\bf x})$ in \eqref{PF_NL} and $P^k_{i(j)}({\bf x})$ in \eqref{GP_SinM_NL}. 
Both are characterized by the partial derivative of a Koopman eigenfunction $\phi_j$ with respect to $x_k$. 
Since $V_{{jk}}$ in \eqref{PF_NL} and $V_{{ik}}$ in \eqref{GP_SinM_NL} are constants, the state-dependency of $P_j^k$ and $P_{i(j)}^k$ is the same. 
The Koopman modes $V_{{jk}}$ and $V_{{ik}}$ determine the amplification and phase (argument) shift applied to the common partial derivative. 
\end{remark}

Now, let ${\rm d}x_k(t; {\bf x}^0, {\delta{\bf x}})$ denote the variation of the $k$-th state in the solution of the nonlinear system \eqref{Nonlinear_ODE} under an initial change $\delta{\bf x}\in\mathbb{R}^n$, defined as
\begin{equation}
{\rm d}x_k(t; {\bf x}^0, {\delta{\bf x}})
:= \bar{x}_k(t; {\bf x}^0+{\delta{\bf x}})-\bar{x}_k(t; {\bf x}^0).
\label{x{j_n}onlinear}
\end{equation}
Also, let ${\rm d}z_j(t; {\bf x}^0, {\delta{\bf x}})$ denote the variation of the $j$-th mode, defined as
\begin{equation}
{\rm d}z_j(t; {\bf x}^0, {\delta{\bf x}}) := \bar{z}_j(t; {\bf x}^0+{\delta{\bf x}}) - \bar{z}_j(t; {\bf x}^0),
\label{z{j_n}onlinear}
\end{equation}
where the mode and its time evolution are as in Lemma~\ref{lemma:MV_EP}. 
Theorem \ref{DEF_KMD_PF} below provides that the PFs and GPs in Definitions~\ref{def_PFGP_NL} and \ref{def_PFGP_NL_High} characterize the variational dynamics for the nonlinear system \eqref{Nonlinear_ODE}, which is parallel to the linear case in Proposition~\ref{prop:Variation_LTI}. 
\begin{theorem}
\label{DEF_KMD_PF}
Let the nonlinear system \eqref{Nonlinear_ODE} {possess} a stable EP ${\bf x}^\ast$ with a basin of attraction $\mathbb{B}({\bf x}^\ast)$. 
For all ${\bf x}^0\in\mathbb{B}({\bf x}^\ast)$, the variational dynamics \eqref{x{j_n}onlinear} of the system in the $k$-th state $x_k$ under an infinitesimal change ${\delta{\bf x}}$ from the initial state ${\bf x}^0$ are represented as follows:
\begin{align}
{\rm d}x_k(t; {\bf x}^0, {\delta{\bf x}})
&\approx \sum^{n}_{j=1}{\rm e}^{\lambda_jt} {P_{j}^{k}}({\bf x}^0) \delta x_k \nonumber\\
& +\sum_{\substack{\ell = 1\\\ell \neq k}}^{n} 
\sum_{j=1}^{n} {\rm e}^{\lambda_j t} P_{j}^{k(\ell)}({\bf x}^0) \delta x_\ell  \nonumber  \\
& +\sum^\infty_{\substack{{j_1},\ldots,{j_n}\in\mathbb{N}_0 \\ {j_1}+\cdots+{j_n}>1}} 
{\rm e}^{({j_1}\lambda_1+\cdots+{j_n}\lambda_n)t}   
P_{\langle j_1\cdots j_n\rangle}^{k}({\bf x}^0) \delta x_k  \nonumber  \\
& +\sum_{\substack{\ell = 1\\\ell \neq k}}^{n} 
\sum^\infty_{\substack{{j_1},\ldots,{j_n}\in\mathbb{N}_0 \\ {j_1}+\cdots+{j_n}>1}}
{\rm e}^{({j_1}\lambda_1+\cdots+{j_n}\lambda_n)t}
P_{\langle j_1\cdots j_n\rangle}^ {k(\ell)}({\bf x}^0) \delta x_\ell,
\label{dx}
\end{align}
where ${P_{j}^{k}}({\bf x}^0)$ and $P_{j}^{k(\ell)}({\bf x}^0)$ are as in Definition~\ref{def_PFGP_NL}, and $P_{\langle j_1\cdots j_n\rangle}^{k}({\bf x}^0)$ and $P_{\langle j_1\cdots j_n\rangle}^ {k(\ell)}({\bf x}^0)$ are as in Definition~\ref{def_PFGP_NL_High}. 
Furthermore, the variational dynamics \eqref{z{j_n}onlinear} in the $j$-th mode $z_j$ are represented as follows:
\begin{align}
{\rm d}z_j(t; {\bf x}^0, {\delta{\bf x}})
&\approx  
{\rm e}^{\lambda_j t} \sum_{k=1}^{n}  {P_{j}^{k}} ({\bf x}^0) \delta z_j^0
\notag \\
& + {\rm e}^{\lambda_j t} \sum_{\substack{i = 1\\i \neq j}}^{n} \sum_{k=1}^{n} 
P_{i(j)}^{k}({\bf x}^0) \delta z_i^0
\notag \\
& + {\rm e}^{\lambda_j t}  \sum^\infty_{\substack{{j_1},\ldots,{j_n}\in\mathbb{N}_0 \\ {j_1}+\cdots+{j_n}>1}} 
\sum_{k=1}^{n}
P_{\langle j_1\cdots j_n\rangle (j)}^{k}({\bf x}^0) \delta z_{\langle j_1\cdots j_n\rangle}^0,
\label{dz}
\end{align}
where $P_{i(j)}^{k}({\bf x}^0)$ is as in Definition~\ref{def_PFGP_NL}, $P_{\langle j_1\cdots j_n\rangle (j)}^{k}({\bf x}^0)$ is as in Definition~\ref{def_PFGP_NL_High}, and $\delta z_j^0:=\phi_j({\bf x}^0+\delta{\bf x})-\phi_j({\bf x}^0)$
and 
$\delta z_{\langle j_1\cdots j_n\rangle}^0:=\phi_{\langle j_1\cdots j_n\rangle}({\bf x}^0+\delta{\bf x})-\phi_{\langle j_1\cdots j_n\rangle}({\bf x}^0)$
represent the initial changes to the modes associated with the initial change $\delta{\bf x}$ at the state ${\bf x}^0$. 
\end{theorem}
\begin{proof}
See Appendix~\ref{app:DEF_KMD_PF}.
\end{proof}
\begin{remark}
\label{Remark:Linear}
Clearly, the variational dynamics \eqref{dx} and \eqref{dz} are close to \eqref{dx_LTI} and \eqref{dz_LTI} for the LTI system \eqref{LTI}. 
In fact, Proposition~\ref{prop:Variation_LTI} for the LTI system is a special case of Theorem~\ref{DEF_KMD_PF}.  
It is shown in \cite{Mezic2013} that by assuming $n$ distinct eigenvalues $\lambda_j$ of $\bf A$ with left eigenvectors ${\bf u}_j$, {the $n$ pairs of} $\lambda_j$ and $\phi_j({\bf x}) = {\bf u}_j^\top {\bf x}$ ($j=1,\ldots,n$) correspond to the pairs of Koopman principal eigenvalues and eigenfunctions for the LTI system. 
Equation~\eqref{xk_KMD_LTI} is thus exactly the KMD for the LTI system, which is a finite series. 
The Koopman mode ${\bf V}_j$ corresponds to the right eigenvector ${\bf v}_j$ of $\bf A$. 
Here, for the LTI system, the higher-order terms indexed by $\langle j_1\cdots j_n\rangle$ disappear in \eqref{dx} and \eqref{dz}. 
Furthermore, for the LTI system, the PFs and GPs in Definition~\ref{def_PFGP_NL} become \begin{align}
{P_{j}^{k}}({\bf x})& = \frac{\partial({\bf u}_j^\top{\bf x})}{\partial x_k}v_{{jk}} = u_{{jk}}v_{{jk}}, \notag\\
P_{j}^{k(\ell)}({\bf x})& = \frac{\partial({\bf u}_j^\top{\bf x})}{\partial x_\ell}v_{{jk}} = u_{{j\ell}}v_{{jk}}, \notag\\
P_{i(j)}^{k}({\bf x})& = \frac{\partial({\bf u}_j^\top{\bf x})}{\partial x_k}v_{{ik}} = u_{{jk}}v_{{ik}},\notag
\end{align}
and are identical to the ones in Definitions~\ref{def_MinSPF_Linear} to \ref{def_SinMGP_Linear}. 
The state-dependent property disappears by the partial derivative of the state. 
Note that, for the LTI system, the state-in-mode GP $P_{i(j)}^{k}({\bf x})$ does not appear in the evolution of ${\rm d}z_j(t)$. 
Also, the initial change{s} $\delta z_j^0$ in the modes, which are dependent on ${\bf x}^0$ in the nonlinear case, are rendered to the linear case \eqref{dz_LTI}. 
\end{remark}
\begin{remark}
The variational dynamics \eqref{dx} and \eqref{dz} are ${\bf x}^0$-dependent and thus \emph{non-uniform} over the domain of attraction. 
This is intrinsic to the nonlinear vector field, where a mapping that assigns each state to a vector in the corresponding tangent space is \emph{not linear}, and is clearly different from the linear vector field as pointed out in Remark~\ref{remark:uniform_LTI}. 
Our analysis shows that the concept of PFs and GPs characterizes the variational dynamics for both linear and nonlinear vector fields. 
The state-dependent definitions of PFs and GPs in this paper are consistent with the intrinsic property of the nonlinear vector field.
The state's dependency of PFs and GPs over the domain of attraction globally is guaranteed with the Koopman operator framework.
\end{remark}
\begin{remark}
The dynamical meaning of PFs and GPs in the nonlinear case parallels that of the LTI system in Remark~\ref{remark:Physical_LTI}. 
The mode-in-state PF ${P_{j}^{k}}({\bf x}^0)$ and mode-in-state GP $P_{j}^{k(\ell)}({\bf x}^0)$ are interpreted as metrics that quantify the contribution of the $j$-th  mode in the variational dynamics of the $k$-th state element $x_k$ \emph{around ${\bf x}={\bf x}^0$}.
The mode-in-state PF ${P_{j}^{k}}({\bf x}^0)$ is associated with the change in $x_k$, and mode-in-state GP $P_{j}^{k(\ell)}({\bf x}^0)$ is associated with the change in $x_\ell$, not restricted at $x_k$.
For $\ell = k$, $P_{j}^{k(k)}({\bf x}^0)$ becomes identical to ${P_{j}^{k}}({\bf x}^0)$.
Conversely, the state-in-mode PF ${P_{j}^{k}}({\bf x}^0)$ is interpreted as a metric that quantifies the contribution of the $k$-th state element $x_k$ in the variational dynamics of the $j$-th mode $z_j$ \emph{around $z_j=\phi_j({\bf x}^0)$}. 
The state-in-mode GP $P_{i(j)}^k({\bf x}^0)$ is interpreted as a metric that evaluates the contribution of $i$-th mode $z_i$ in the evolution of $z_j$, under the initial change in $x_k$. 
For $i = j$, $P_{j(j)}^{k}({\bf x}^0)$ becomes identical to ${P_{j}^{k}}({\bf x}^0)$.
As a new finding in the nonlinear system, it becomes clear that the state-in-mode GP $P_{i(j)}^k$ appears in the variational dynamics \eqref{dz} of the mode $z_j$ (see the second and third terms on the right-hand side), whereas the state-in-mode GP $P_{i(j)}^k$ does not appear in the LTI case \eqref{zj}.
\end{remark}

Above, we considered the PFs and GPs defined globally in the basin of attraction for the nonlinear system with a stable EP. 
The dynamics around EPs were the target of the preceding works \cite{Sanchez-Gasca2005, Hamzi2014, Tian2018, Hamzi2020, Netto2019}. 
Below, we describe the PFs and GPs defined globally (in the sense of a domain of attraction) for the nonlinear system with a stable LC, which has not been targeted.
The developing theory here parallels the case of stable EP according to the Koopman operator framework \cite{Mezic2020, Susuki2021}. 
The next two lemmas clarify the existence of modal variables for a nonlinear system with a stable LC. 
\begin{lemma}
\label{lemma:KE_LC}
Let the nonlinear system \eqref{Nonlinear_ODE} possess a stable LC, denoted as $\Gamma$, with the basic angular frequency $\Omega$ and a basin of attraction $\mathbb{B}({\Gamma}) \subseteq {\mathbb R}^n$. 
Let the $n-1$ characteristic multipliers $\mu_j\in\mathbb{C}$ ($j=2,\ldots,n$) of $\Gamma$ be distinct, strictly positive, and sorted so that $1>\mu_2\geq \mu_3 \geq \cdots \geq \mu_{n}$.  
Any $\nu_j$ satisfying $\mu_j={\rm e}^{2\pi\nu_j/\Omega}$ is called a characteristic exponent.
Under a suitable choice of the space $\mathscr{F}$ of observables $f: \mathbb{B}({\Gamma})\to\mathbb{C}$, the Koopman operator associated with \eqref{Nonlinear_ODE} possesses only the point spectrum, containing the $n$ Koopman principal eigenvalues $\lambda_j$ defined as
\[
\lambda_1 = {\rm i}\Omega, \quad \lambda_j = \nu_j \quad (j=2,\ldots,n).
\] 
For every $\lambda_j$, there exist associated Koopman eigenfunctions $\phi_j\in\mathscr{F}\setminus\{0\}$ that are smooth over $\mathbb{B}(\Gamma)$. 
Furthermore, there exist Koopman eigenvalues, generally expressed as 
\begin{align}
\lambda_{\langle j_1\cdots j_n\rangle} 
&= {\rm i} j_1 \Omega + (j_2\nu_2+\cdots+j_n\nu_{n}) \nonumber \\
&=  j_1 \lambda_1 + j_2\lambda_2+\cdots+j_n\lambda_{n}, 
\label{eqn:KE_LC}
\end{align}
with
\begin{equation}
(j_1,\ldots,j_n)\in \mathcal{I}:=\mathcal{I}_+\cup \mathcal{I}_{-}, \nonumber
\end{equation}
where the index set $\mathcal{I}_+$ is $\{(j_1,\ldots,j_n)\in\mathbb{N}^n_0~:~j_1+\cdots+j_n > 1\}$ and $\mathcal{I}_{-}$ is 
$\{(j_1,j_2,\ldots,j_n)\in\mathbb{Z}\times\mathbb{N}^{n-1}_0~:~j_1=-n,~n\in\mathbb{N}\}$. 
{And, the} associated Koopman eigenfunctions $\phi_{\langle j_1\cdots j_n\rangle} \in\mathscr{F}\setminus\{0\}$, defined in the same manner as in Lemma~\ref{lemma:KEF_EP}, are smooth over $\mathbb{B}(\Gamma)$. 
\end{lemma}
\begin{proof}
See Theorem~8.1 of \cite{Mezic2020}. 
The way of introducing $n$ Koopman principal eigenvalues is from \cite{Mauroy2018}. 
\end{proof}
\begin{lemma}
\label{lemma:mode_LC}
The Koopman eigenfunctions $\phi_j$ and $\phi_{\langle j_1\cdots j_n\rangle}$ in Lemma~\ref{lemma:KE_LC} provide modal variables for the nonlinear system \eqref{Nonlinear_ODE} with the stable LC $\Gamma$. 
That is, by defining $z_j:=\phi_j({\bf x})$ and $z_{\langle j_1\cdots j_n\rangle}:=\phi_{\langle j_1\cdots j_n\rangle}({\bf x})$, their time evolutions $\bar{z}(t, {\bf x}^0)=\phi_j({\bf S}^t({\bf x}^0))$ and $\bar{z}_{\langle j_1\cdots j_n\rangle}(t; {\bf x}^0)=\phi_{\langle j_1\cdots j_n\rangle}({\bf S}^t({\bf x}^0))$ are represented in the following manner: for $t\geq 0$,
\begin{align}
\bar{z}_{j}(t; {\bf x}^0) 
&= {\rm e}^{\lambda_jt} \bar{z}_{j}(0; {\bf x}^0) 
= {\rm e}^{\lambda_jt} \phi_{j}({\bf x}^0), \notag\\
\bar{z}_{\langle j_1\cdots j_n\rangle}(t; {\bf x}^0) 
&= {\rm e}^{(  j_1 \lambda_1 + \cdots + j_n\lambda_{n} )t} \bar{z}_{\langle j_1\cdots j_n\rangle}(0; {\bf x}^0) \notag \\
&= {\rm e}^{(  j_1 \lambda_1 + \cdots + j_n\lambda_{n} )t} \phi_{\langle j_1\cdots j_n\rangle}({\bf x}^0). 
\label{LC_Mode}
\end{align}
\end{lemma}
\begin{proof}
The proof is straightforward and thus omitted.
\end{proof}

Theorem \ref{Case_LC} below provides that the PFs and GPs proposed in this paper globally characterize the variational dynamics \eqref{x{j_n}onlinear} and \eqref{z{j_n}onlinear} for the nonlinear system \eqref{Nonlinear_ODE} with $\Gamma$. 
This implies that the common PFs and GPs work for {the} nonlinear systems with stable EPs and LCs. 
\begin{theorem}
\label{Case_LC}
Let the nonlinear system \eqref{Nonlinear_ODE} {possess} a stable LC $\Gamma$, with a basin of attraction $\mathbb{B}({\Gamma})$. 
For all ${\bf x}^0\in\mathbb{B}(\Gamma)$, the variational dynamics \eqref{x{j_n}onlinear} and \eqref{z{j_n}onlinear} in the $k$-th state $x_k$ and the $j$-th mode $z_j$ under an infinitesimal change $\delta {\bf x}$ from the initial state ${\bf x}^0$ are represented as follows:
\begin{align}
{\rm d}x_k(t; {\bf x}^0, {\delta{\bf x}}) 
&\approx \sum_{j=1}^{n} {\rm e}^{\lambda_j t} P^k_j({\bf x}^0) \delta x_k 
\notag \\
&+ \sum_{\substack{\ell = 1\\\ell \neq k}}^{n} 
\sum_{j=1}^{n} {\rm e}^{\lambda_j t} P^{k (\ell)}_j({\bf x}^0) \delta x_\ell
\notag \\
&+ \sum_{\mathcal{I}} 
{\rm e}^{(j_1 \lambda_1 +\cdots+j_n\lambda_{n})t}
P^k_{\langle j_1\cdots j_n\rangle}({\bf x}^0)  \delta x_k  \nonumber  \\
&+ \sum_{\substack{\ell = 1\\\ell \neq k}}^{n} 
\sum_{\mathcal{I}} 
{\rm e}^{(j_1 \lambda_1 +\cdots+j_n\lambda_{n})t}
P^{k (\ell)}_{\langle j_1\cdots j_n\rangle}({\bf x}^0) \delta x_\ell,
\label{dx_LC}
\end{align}
\begin{align}
{\rm d}z_j(t; {\bf x}^0, {\delta{\bf x}}) 
&\approx  
\sum_{k=1}^{n} {\rm e}^{\lambda_j t} P^k_j({\bf x}^0)  \delta z_j^0
\notag \\
&+ \sum_{\substack{i = 1\\i \neq j}}^{n} \sum_{k=1}^{n} {\rm e}^{\lambda_j t}
P^k_{i (j)}({\bf x}^0) \delta z_i^0
\notag \\
&+\sum_{\mathcal{I}}  \sum_{k=1}^{n}
{\rm e}^{\lambda_j t}
P^k_{\langle j_1\cdots j_n\rangle (j)}({\bf x}^0) 
\delta z_{\langle j_1\cdots j_n\rangle}^0,
\label{dz_LC}
\end{align}
where $P^k_j({\bf x}^0)$, $P_{j}^{k(\ell)}({\bf x}^0)$, and $P^k_{i(j)}({\bf x}^0)$ are introduced in Definition~\ref{def_PFGP_NL}, $P^k_{\langle j_1\cdots j_n\rangle}({\bf x}^0)$, $P^{k (\ell)}_{\langle j_1\cdots j_n\rangle}({\bf x}^0)$, and $P^k_{\langle j_1\cdots j_n\rangle  (j)}({\bf x}^0)$ are in Definition~\ref{def_PFGP_NL_High},  and $\delta z_j^0=\phi_j({\bf x}^0+\delta{\bf x})-\phi_j({\bf x}^0)$ and $\delta z_{\langle j_1\cdots j_n\rangle}^0=\phi_{\langle j_1\cdots j_n\rangle}({\bf x}^0+\delta{\bf x})-\phi_{\langle j_1\cdots j_n\rangle}({\bf x}^0)$ represent the initial changes to the modes associated with the initial change $\delta{\bf x}$ at the state ${\bf x}^0$.
\end{theorem}
\begin{proof}
See Appendix~\ref{app_LC}. 
\end{proof}

\section{Illustrative Examples}
\label{sec:examples}

This section presents illustrative examples of the PFs and GPs by applying them to simple two-dimensional models of nonlinear systems. 

\subsection{Nonlinear system with stable equilibrium point} 
\label{EX1}
 
Following \cite{Netto2019, Brunton2016}, let us consider a two-dimensional nonlinear system with a globally stable EP,
\begin{align}
\left\{
\centering
\begin{array}{l}
\dot{x}_1 =  -x_1 + x_2^2, \\ 
\dot{x}_2 =  - \sqrt{2}x_2.
\end{array}
\right. 
\label{EX1_System}
\end{align}
The EP is located at the origin and the eigenvalues of the {system's Jacobian matrix} are $\lambda_1=-1$ and $\lambda_2=-\sqrt{2}$.
The nonlinear system \eqref{EX1_System} is widely utilized in the Koopman operator framework where Koopman eigenfunctions and associated modes are analytically derived, implying that it works as an introductory example of our theory.
In this case, the two Koopman principal eigenvalues and associated Koopman eigenfunctions are
\begin{equation}
\left\{
\begin{array}{cl}
\lambda_1=-1, & \phi_1({\bf x})=x_1+\frac{1+2\sqrt{2}}{7}x_2^2,\\
\lambda_2=-\sqrt{2}, & \phi_2({\bf x})=x_2{.}
\end{array}
\right.
\label{KEF_1}
\end{equation}
where ${\bf x}=(x_1,x_2)^\top$. 
Also, as one higher-order mode, the following pair of Koopman eigenvalue and eigenfunction with ${j_1}=0$ and ${j_2}=2$ is found: 
\begin{equation}
{j_1}\lambda_1+{j_2}\lambda_2=-2\sqrt{2}, \quad 
\phi_1^{j_1}\phi_2^{j_2}({\bf x})=x_2^2.
\label{KEF_2}
\end{equation}
Then, the KMD for the states of \eqref{EX1_System} is derived as follows:
\begin{align} 
\begin{bmatrix}
\bar{x}_1(t; {\bf x}^0) \\
\bar{x}_2(t; {\bf x}^0) \\
\end{bmatrix}
&= {\rm e}^{\lambda_1 t}\phi_{1}({\bf x}^0) 
\begin{bmatrix}
1  \\
0 \\
\end{bmatrix}
+ {\rm e}^{\lambda_2 t}\phi_{2}({\bf x}^0) 
\begin{bmatrix}
0  \\
1 \\
\end{bmatrix} 
+ {\rm e}^{(0\lambda_1+2\lambda_2)t}{\phi^0_{1}\phi^2_{2}}({\bf x}^0) 
\begin{bmatrix}
\frac{-1-2\sqrt{2}}{7}  \\
0 \\
\end{bmatrix}, \notag
\end{align}
where the vectors on the right-hand side are the Koopman modes ${\bf V}_1$, ${\bf V}_2$, and ${\bf V}_{{\langle 02\rangle}}$. 

Now, the PFs and GPs are analytically derived for the nonlinear system \eqref{EX1_System}.
By using \eqref{PF_NL} and \eqref{PF_NL_High}, the PFs ${P_{j}^{k}}$ between the $j$-th mode and the $k$-th state as well as ${P_{{\langle j_1j_2\rangle}}^{k}}$ for the higher-order mode are the following:
\begin{equation} 
\begin{bmatrix}
{P_{1}^{1}}({\bf x}) & {P_{1}^{2}}({\bf x})\\
{P_{2}^{1}}({\bf x}) & {P_{2}^{2}}({\bf x})\\
{P_{{\langle 02\rangle}}^{1}}({\bf x}) & {P_{{\langle 02\rangle}}^{2}}({\bf x})\\
\end{bmatrix}
=
\begin{bmatrix}
1 & 0 \\
0 & 1 \\
0 & 0
\end{bmatrix}.\label{EX1_PF}
\end{equation}
Here, the PFs in \eqref{EX1_PF} can be interpreted as either mode-in-state PF and state-in-mode PF.
The PF ${P_{1}^{1}}({\bf x})=1$ implies that under an infinitesimal change in $x_1$, the first mode appears in the variation ${\rm d}x_1(t)$. 
The same implication holds for ${P_{2}^{2}}({\bf x})=1$ regarding ${\rm d}x_2(t)$. 
The PF ${P_{2}^{1}}({\bf x})=0$ implies that under a change in $x_1$, the second mode does not appear in ${\rm d}x_1(t)$. 
No state dependence appears in the PFs. 
These are because $\phi_1({\bf x})$ is linear in $x_1$ and $\phi_2({\bf x})$ is linear in $x_2$, respectively.
{Note that \eqref{EX1_PF} has} no information about the higher-order mode.

Next, by using \eqref{GP_MinS_NL} and \eqref{GP_MinS_NL_High}, the mode-in-state GPs are the following: 
\[
\begin{bmatrix}
P_{1}^{1(2)}({\bf x}) & P_{1}^{2(1)}({\bf x}) \\\noalign{\vskip 1mm}
P_{2}^{1(2)}({\bf x}) & P_{2}^{2(1)}({\bf x}) \\\noalign{\vskip 1mm}
P_{{\langle 02\rangle}}^{1(2)}({\bf x}) & P_{{\langle 02\rangle}}^{2(1)}({\bf x})\\ 
\end{bmatrix}
=
\begin{bmatrix}
\frac{2+4\sqrt{2}}{7}x_2 & 0 \\
0 & 0 \\
\frac{-2-4\sqrt{2}}{7}x_2 & 0
\end{bmatrix}.
\]
The mode-in-state GPs $P_{1}^{1(2)}({\bf x}) = \frac{2+4\sqrt{2}}{7}x_2$ and $P_{{\langle 02\rangle}}^{1(2)}({\bf x}) = \frac{-2-4\sqrt{2}}{7}x_2$ implies that under an infinitesimal change in $x_2$, both the first mode and higher-order mode appear in the variation ${\rm d}x_1(t)$.
Also, $P_{1}^{1(2)}({\bf x})$ and $P_{{\langle 02\rangle}}^{1(2)}({\bf x})$ are linear functions in $x_2$ and are not uniform, indicating their state-dependency on $x_2$, which stems from the nonlinear terms in the eigenfunctions $\phi_1({\bf x})$ and ${\phi^0_1\phi^2_2}({\bf x})$.
Besides, the mode-in-state GP{s} $P_{1}^{2(1)}({\bf x}) = P_{2}^{2(1)}({\bf x})  =P_{{\langle 02\rangle}}^{2(1)}({\bf x}) = 0$ imply that excited modes do not appear in ${\rm d}x_2(t)$, under an infinitesimal change in $x_1$. 

Finally, using \eqref{GP_SinM_NL} and \eqref{GP_SinM_NL_High}, the state-in-mode GPs are derived as follows:  
\begin{align} 
\begin{bmatrix}
   &  P_{2(1)}^{1}({\bf x}) &  P_{{\langle 02\rangle}(1)}^{1}({\bf x})\\  \noalign{\vskip 1mm}
P_{1 (2)}^{1}({\bf x})& &P_{{\langle 02\rangle}(2)}^{1}({\bf x}) \\ \noalign{\vskip 1mm}
P_{1 ({\langle 02\rangle})}^{1}({\bf x}) & P_{2({\langle 02\rangle})}^{1}({\bf x}) &
\end{bmatrix}
&=
\begin{bmatrix}
          & 0  & \frac{-1-2\sqrt{2}}{7} \\
0        &     & 0 \\
0        & 0  &
\end{bmatrix}. \nonumber \\
\begin{bmatrix}
   &  P_{2(1)}^{2}({\bf x}) &  P_{{\langle 02\rangle}(1)}^{2}({\bf x})\\  \noalign{\vskip 1mm}
P_{1 (2)}^{2}({\bf x})& &P_{{\langle 02\rangle}(2)}^{2}({\bf x}) \\ \noalign{\vskip 1mm}
P_{1 ({\langle 02\rangle})}^{2}({\bf x}) & P_{2({\langle 02\rangle})}^{2}({\bf x}) &
\end{bmatrix}
&=
\begin{bmatrix}
          & \frac{2+4\sqrt{2}}{7}x_2 & 0 \\
0        &                   & 0 \\
0        & {2x_2}  &
\end{bmatrix}, \nonumber 
\end{align}
The state-in-mode GPs $P_{2(1)}^{2}({\bf x}) = \frac{2+4\sqrt{2}}{7}x_2$ and $P_{2({\langle 02\rangle})}^{2}({\bf x}) = 2x_2$ imply that under an infinitesimal change in $x_2$, the first and higher-order modes are excited.
The evolutions of ${\rm d}z_1$ and ${\rm d}z_{\langle 02\rangle}$ are affected by the second mode.
The connection of the excitation of $\phi_2({\bf x})  = x_2$ to the first one $\phi_1({\bf x})$ and the higher-order one ${\phi_1^0 \phi_2^2}({\bf x})$ stems from its presence in $\phi_1({\bf x})=x_1+\frac{1+2\sqrt{2}}{7}x_2^2$ and ${\phi_1^0 \phi_2^2}({\bf x})=x^2_2$. 
See \eqref{KEF_1} and \eqref{KEF_2}.  
Also, $P_{2(1)}^{2}({\bf x})$ and $P_{2({\langle 02\rangle})}^{2}({\bf x})$ are not uniform, {which is a clear} indication that the magnitude of their effect depends on state $x_2$.
Besides, $P_{{\langle 02\rangle}(1)}^{1}({\bf x}) = \frac{-1-2\sqrt{2}}{7}$ implies that under an infinitesimal change in $x_1$, the first mode is excited, and the evolution of ${\rm d}z_1$ is affected by the higher-order mode.
However, from \eqref{KEF_2}, under an infinitesimal change in only $x_1$, note that the higher-order mode is not excited, and ${\rm d}z_1$ is not affected.

\subsection{Nonlinear system with stable limit cycle} 
\label{EX2}

\begin{figure*}[h]
  \begin{minipage}{0.45\linewidth}
    \centering
    \includegraphics[width=\hsize]{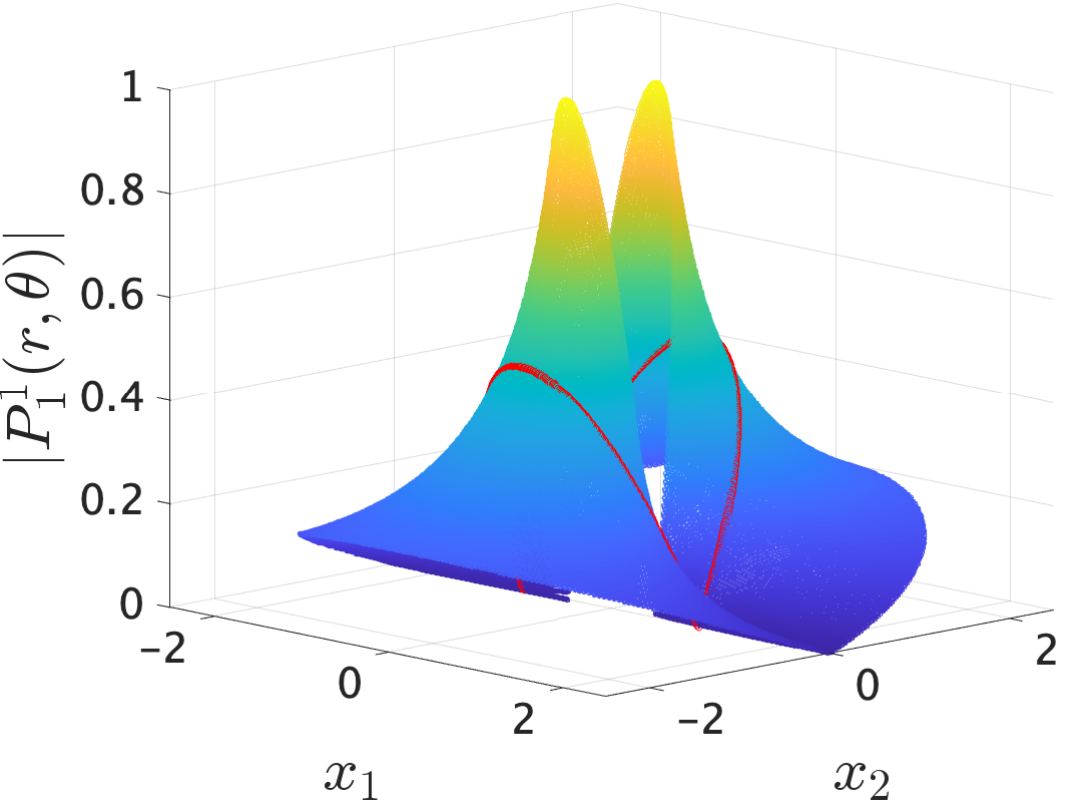}
    \subcaption{$P^{1}_1(r,\theta)$}
    \label{Fig:EX2_PF}
  \end{minipage}
  \begin{minipage}{0.45\linewidth}
    \centering
    \includegraphics[width=\hsize]{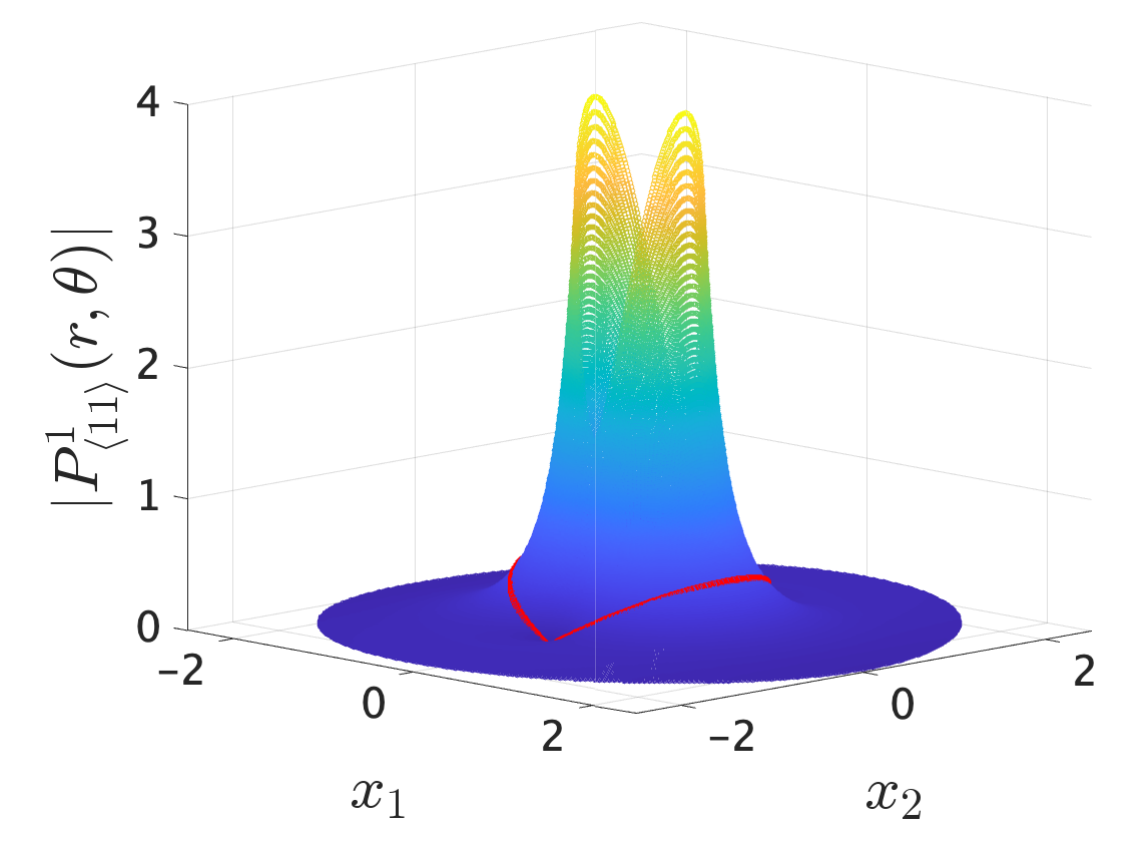}
    \subcaption{$P^1_{{\langle 11\rangle}}(r,\theta)$}
     \label{Fig:EX2_PF_Complex}
  \end{minipage}
    \begin{minipage}{0.45\linewidth}
    \centering
    \includegraphics[width=\hsize]{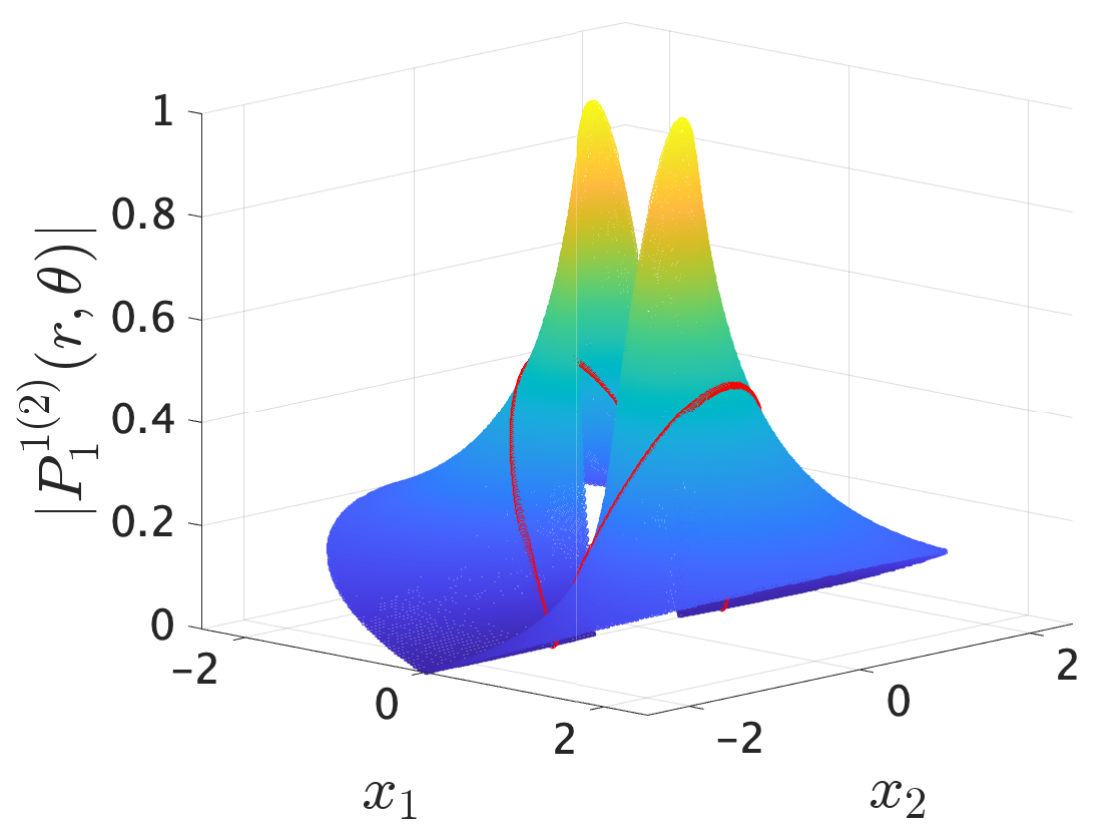}
    \subcaption{$P^{1(2)}_1(r,\theta)$}
     \label{Fig:EX2_GP_MinS}
  \end{minipage}
  \caption{
Participation Factors (PFs) $P^{1}_1(r,\theta)$ and $P^1_{{\langle 11\rangle}}(r,\theta)$, and the mode-in-state Generalized Participation (GP) $P^{1(2)}_1(r,\theta)$ for the nonlinear system \eqref{EX2_System} with the stable Limit Cycle (LC). 
The figures show the absolute values of the PFs and mode-in-state GP, and the \emph{red} curves show the location of the LC.
The PF $P^{1}_1(r,\theta)$ indicates the contribution of the first mode (pure rotational component) to the state $x_1$ under an infinitesimal change in $x_1$.  
The PF $P^1_{{\langle 11\rangle}}(r,\theta)$ indicates the contribution of the higher-order complex 
mode to the state $x_1$ under an infinitesimal change in $x_1$.  
The mode-in-state GP $P^{1(2)}_1(r,\theta)$ indicates the contribution of the first mode (pure rotational component) to the state $x_1$ under an infinitesimal change in $x_2$.
}
\label{Figs_EX2}
\end{figure*}

Now, let us consider a two-dimensional nonlinear system with a stable LC,
\begin{align}
  \left\{
  \centering
    \begin{array}{l}
        \dot x_1 =  x_1 - x_2 - x_1 (x_1^2 + x_2^2),   \\\noalign{\vskip 1mm}
        \dot x_2 =  x_1 + x_2 - x_2 (x_1^2 + x_2^2){.} 
    \end{array}
  \right.
  \label{EX2_System}
\end{align}
The origin is an unstable EP, and the unit circle $x^2_1+x^2_2=1$ corresponds to the stable LC $\Gamma$. 
By using a polar coordinate transformation $(r,\theta)\in(0,\infty)\times \mathbb{R}/(2\pi\mathbb{Z})$ via $x_1=r{\cos}{\theta}, x_2=r{\sin}{\theta}$, \eqref{EX2_System} is written as follows:
\begin{align}
  \left\{
  \centering
    \begin{array}{l}
        \dot r =  r - r^3,  \\
        \dot \theta =  1.    
    \end{array}
  \right.
  \label{EX2_Polar}
\end{align}
For this system, $r=1$ coincides with $\Gamma$, ``$-2$" is the linearized eigenvalue corresponding to the negative characteristic exponent $\nu_2$, and ``$1$" is the basic angular frequency $\Omega$ of $\Gamma$.  
By hand, the two Koopman principal eigenvalues and associated Koopman eigenfunctions in the polar coordinate are derived as
\[
\left\{
\begin{array}{ll}
\lambda_1={\rm i}, & \phi_1(r,\theta)={\rm e}^{{\rm i}\theta}, \\
\lambda_2=-2, & \displaystyle \phi_2(r,\theta) = \frac{r^2-1}{r^2}.
\end{array}
\right.
\]
Also, as one higher-order mode, the following pair of Koopman eigenvalue and eigenfunction with $j_1=j_2=1$ is found:
\begin{equation}
j_1\lambda_1+j_2\lambda_2={\rm i}-2, \quad 
{\phi^{j_1}_1\phi^{j_2}_2}({\bf x})={\rm e}^{{\rm i}\theta}\frac{r^2-1}{r^2}.
\end{equation}
As shown in Lemma~\ref{lemma:mode_LC}, the Koopman eigenfunctions provide modal variables for the nonlinear system \eqref{EX2_System}. 
The first mode $z_1=\phi_1(r,\theta)$ extracts the pure rotational component of the flow along $\Gamma$ with the basic angular frequency $\Omega$.
The second mode $z_2=\phi_2(r,\theta)$  extracts a convergent component of the flow perpendicular to $\Gamma$, and the higher-order mode $z_{{\langle 11\rangle}}={\phi^1_1\phi^1_2}({\bf x})$ exhibits a combination of the above two components.
Here, the flow $r(t)$ in \eqref{EX2_Polar} converges to the EP $r=1$ asymptotically from any $r_0\in(0,\infty)$.
The KMD-like representation of the flow can be given from {\cite{Gaspard1995}} as
\begin{align}
r(t)&= 1
+ \frac{1}{2} \phi_2(r_0,{}_\sqcup) {\rm e}^{\lambda_2t} 
+ \sum^{M}_{j\geq 2} \frac{1}{j!} \frac{(2j - 1)!!}{2^j} \phi_2(r_0,{}_\sqcup)^j {\rm e}^{j\lambda_2 t} 
+ {R(t)}, 
\notag
\end{align}
where $\sqcup$ stands for the arbitrary value of $\theta$, $M$ for a sufficiently large integer, $!!$ for the second factorial, and $R(t)$ for the residual term.
Note that the convergence of the residual term $R(t)$ to $0$ as $M\rightarrow \infty$ is guaranteed if $\frac{1}{\sqrt{2}} < r_0$ holds. 
Now, thanks to ${\rm e}^{{\rm i}\theta(t)}=\phi_1({}_\sqcup,\theta_0){\rm e}^{\lambda_1t}$ from \eqref{EX2_Polar} with an initial $\theta_0$ and $x_1(t)=r(t)\cos\theta(t)=r(t)\{{\rm e}^{{\rm i}\theta(t)}+{\rm e}^{-{\rm i}\theta(t)}\}/2$, the KMD for $x_1$ is approximately derived as follows:
\begin{align}
x_1(t) &= 
{\rm e}^{\lambda_1t}\phi_1(r_0,\theta_0)\frac{1}{2}
+{\rm e}^{(\lambda_1+\lambda_2)t}\phi_1(r_0,\theta_0)\phi_2(r_0,\theta_0)\frac{1}{2^2}
\nonumber\\
& +\sum_{j\geq 2}^{{M}}
{\rm e}^{(\lambda_1+j\lambda_2)t} \phi_1(r_0,\theta_0) 
\phi_2(r_0,\theta_0)^j \frac{1}{j!} \frac{(2j - 1)!!}{2^{j+1}} 
{+R_1(t)},
\label{EX2_x1_KMD}
\end{align}
where $R_1(t)$ again denotes the residual term. 
From $z_{{\langle 1j\rangle}}=\phi_{{\langle 1j\rangle}}(r,\theta)=\phi_1(r,\theta)\phi_2(r,\theta)^j$, $1/2$ and ${(2j - 1)!!}/({j!}2^{j+1})$ are the Koopman modes $V_{1 1}$ and $V_{{{{\langle 1j\rangle}}} 1}$ for $z_1$ and $z_{{{\langle 1j\rangle}}}$ ($j\geq 1$), respectively. 
The other modes, including $z_2$, do not appear in the time evolution of $x_1$ (that is, the corresponding modes are identically zero).  
The KMD for the state $x_2$ is approximately derived in the same manner as above and given by
\begin{align}
x_2(t) &= 
{\rm e}^{\lambda_1t}\phi_1(r_0,\theta_0)\frac{1}{2{\rm i}}
+ {\rm e}^{(\lambda_1+\lambda_2)t}\phi_1(r_0,\theta_0)\phi_2(r_0,\theta_0)\frac{1}{2^2{\rm i}}
\nonumber\\
& +\sum_{j\geq 2}^{{M}}
{\rm e}^{(\lambda_1+j\lambda_2)t}\phi_1(r_0,\theta_0)\phi_2(r_0,\theta_0)^j \frac{1}{j!}
\frac{(2j - 1)!!}{2^{j+1}{\rm i}}
 {+R_2(t),}
\label{EX2_x2_KMD}
\end{align}
with the residual term $R_2(t)$.

First, the PFs for the nonlinear system with $\Gamma$ are derived. 
By using \eqref{PF_NL} and \eqref{PF_NL_High} with the chain rule of differentiation, the PFs $P^k_1$ between the first mode and the $k$-th state as well as $P^k_{\langle11\rangle}$ for the higher-order mode are the following:
\begin{equation}
\left\{
\begin{aligned}
{P^1_{1}}(r,\theta) 
&= -{\rm i}\frac{\sin \theta}{2r}{\rm e}^{{\rm i}\theta}, \\
{P^2_{1}}(r,\theta)
&= \frac{\cos \theta}{2r}{\rm e}^{{\rm i}\theta},\\
{P^1_{{\langle 11\rangle}}}(r,\theta) 
&= \frac{{\rm e}^{ {\rm i} \theta}} {4r^3} \{2\cos \theta  - {\rm i}(r^2-1)\sin \theta \}, \\
{P^2_{{\langle 11\rangle}}}(r,\theta) 
&= -{\rm i}\frac{{\rm e}^{ {\rm i} \theta}} {4r^3} \{2\sin \theta + {\rm i}(r^2-1)\cos \theta\}.
\end{aligned}
\right.
\label{PFs_Example_LC}
\end{equation}
Here, the PFs in \eqref{PFs_Example_LC} can be interpretated as either mode-in-state PF and state-in-mode PF.
Note that $P^k_2(r,\theta)=0$ {holds} because the second mode does not appear in the KMDs \eqref{EX2_x1_KMD} {and} \eqref{EX2_x2_KMD}. 
Here, let us focus on $P^1_1(r,\theta)$. 
The absolute value of $P^1_{1}(r,\theta)$ is shown in Fig.~\ref{Fig:EX2_PF}.
The PF $P^1_{1}(r,\theta)$ indicates the contribution of the first mode (pure rotational component) to the state $x_1$ under an infinitesimal change in $x_1$. 
We see $P^1_{1}(r,\theta)$ is zero at $x_2 = 0$, implying that under a change in $x_1$, the first mode does not appear in the variation ${\rm d}x_1(t)$. 
That indicates the change of $x_1$ at $x_2=0$ does not affect the evolution of $z_1$.
That is, if $x_1$ infinitesimally changes at $x_2=0$, the rotational angle $\theta$ does not change and no effect appears in the pure rotational component of the flow.
It is because the change of $x_1$ at $x_2=0$ is parallel to the radial direction of the LC, where $r$ (i.e., the distance from the original point) changes directly but the rotational angle $\theta$ does not change at all, causing no effect in $\phi_1(r,\theta)={\rm e}^{{\rm i}\theta}$. 
On the other hand, $P^1_{1}(r,\theta)$ is relatively large at $x_1=0$. 
This implies that under a change in $x_1$, the first mode appears in ${\rm d}x_1(t)$. 
That indicates the change of  $x_1$ at $x_1=0$ affects the evolution of $z_1$.
That is, if $x_1$ infinitesimally changes at $x_1=0$, the rotational angle $\theta$ directly changes, and the pure rotational component of the flow is also affected.  
Then, the associated mode $\phi_1$ of the rotational dynamics appears in the evolution of ${\rm d}x_1(t)$. 
It is because the change of $x_1$ at $x_1=0$ is vertical to the radial direction of the LC, 
where $r$ does not change but the rotational angle $\theta$ changes directly, thereby affecting $\phi_1(r,\theta)={\rm e}^{{\rm i}\theta}$. 
Also, the absolute value of $P^1_{{\langle 11\rangle}}(r, \theta)$ is shown in Fig.~\ref{Fig:EX2_PF_Complex}.
The PF $P^1_{{\langle 11\rangle}}(r, \theta)$ indicates the contribution of the higher-order mode (a complex component of $\phi_1$ and $\phi_2$) to the state $x_1$ under an infinitesimal change in $x_1$. 
Since $P^1_{{\langle 11\rangle}}$ is related to the convergent component of the flow as well as the pure rotational one, it is not equal to zero at $x_2 = 0$, whereas $P^1_{1}$ does.

Next, by using \eqref{GP_MinS_NL} and \eqref{GP_MinS_NL_High}, the mode-in-state GP is the following: 
\begin{equation}
\left\{
\begin{aligned}
P^{1(2)}_1(r,\theta) &= {\rm i}\frac{\cos \theta}{2r}{\rm e}^{{\rm i}\theta}, \\
P^{2(1)}_1(r,\theta) &= - \frac{\sin \theta}{2r}{\rm e}^{{\rm i}\theta},  \\
P^{1(2)}_{{\langle 11\rangle}}(r,\theta)  &= \frac{{\rm e}^{ {\rm i} \theta}} {4r^3}          \{2\sin \theta + {\rm i}(r^2-1)\cos \theta\}, \\
P^{2(1)}_{{\langle 11\rangle}}(r,\theta) &= {-\rm i} \frac{{\rm e}^{ {\rm i} \theta}} {4r^3} \{2\cos \theta - {\rm i}(r^2-1)\sin \theta \}{.}
\end{aligned}
\right. 
\end{equation}
The absolute value of $P^{1(2)}_1(r,\theta)$ is shown in Fig.~\ref{Fig:EX2_GP_MinS}. 
The {mode-in-state} GP $P^{1(2)}_1(r,\theta)$ indicates the contribution of the first mode (pure rotational component) to the state $x_1$ under an infinitesimal change in $x_2$.
Its dynamical meaning can be discussed in the same manner as for $P^1_1(r,\theta)$. 
Here, we have the following relation of the mode-in-state PF and GP:
\[
P^1_j={\rm i}P^{2(1)}_j, \quad 
{\rm i}P^2_j=P^{1(2)}_j, \quad j\in\{1, {\langle 11\rangle}\}.
\]
This relation is natural by considering Remark~\ref{Rem_SinM}, the definition{s} of mode-in-state PF and GP, and the KMDs in \eqref{EX2_x1_KMD} and \eqref{EX2_x2_KMD}.
From Definition~\ref{DEF_KMD_PF}, the state dependency {is the same for} $P_{j}^{k}$ and $P_{j}^{\ell(k)}$, and the difference between $P_{j}^{k}$ and $P_{j}^{\ell(k)}$ for $j\in\{1, {\langle 11\rangle}\}$ is only from the Koopman modes $V_{{jk}}$ and $V_{{j\ell}}$.
Additionally, the decomposition of $\bf x$ is derived by $x_1 = r(t)\cos\theta(t) =r(t)\{{\rm e}^{{\rm i}\theta(t)}+{\rm e}^{-{\rm i}\theta(t)}\}/2$ and $x_2 = r(t)\sin\theta(t) =r(t)\{{\rm e}^{{\rm i}\theta(t)}-{\rm e}^{-{\rm i}\theta(t)}\}/2{\rm i}$.
For this, the Koopman modes are determined to be $V_{{j1}} = {\rm i} V_{{j2}}$, leading to the relation in the above equation.

Finally, using \eqref{GP_SinM_NL} and \eqref{GP_SinM_NL_High}, the state-in-mode GP is derived as follows:
\begin{equation}
\left\{
\begin{aligned}
{P^1_{{\langle 11\rangle}(1)}}(r,\theta) &= -{\rm i}\frac{\sin \theta}{4r}{\rm e}^{{\rm i}\theta}, \\
{P^2_{{\langle 11\rangle}(1)}}(r,\theta) &= \frac{\cos \theta}{4r}{\rm e}^{{\rm i}\theta}, \\
{P^1_{1({\langle 11\rangle})}}(r,\theta) &= \frac{{\rm e}^{{\rm i} \theta}}{2r^3} \{2\cos \theta  - {\rm i}(r^2-1)\sin \theta\}. \\
{P^2_{1({\langle 11\rangle})}}(r,\theta) &= -{\rm i}\frac{{\rm e}^{ {\rm i} \theta}} {2r^3} \{2\sin \theta + {\rm i}(r^2-1)\cos \theta\}.
\end{aligned}
\right.
\end{equation}
For example, let us take the case of $P^1_{{\langle 11\rangle}(1)}(r,\theta)$, which indicates the contribution of the higher-order complex mode to the first mode, under an infinitesimal change in $x_1$. 
Now, we have the following relation for the state-in-mode GP $P^1_{{\langle 11\rangle}(1)}(r,\theta)$ and PF $P^1_{1}(r,\theta)$:
\[
P^1_{{\langle 11\rangle}(1)}(r,\theta) = \frac{1}{2}P^1_1(r,\theta).
\]
The relation is natural from Remark~\ref{Rem_SinM}.
Both $P^1_{{\langle 11\rangle}(1)}(r,\theta)$ and $P^1_{1}(r,\theta)$ are characterized by the partial derivative of Koopman eigenfunction $\phi_1$ by $x_1$, leading to the same state dependency on $\bf x$.
The difference between $P^1_{{\langle 11\rangle}(1)}(r,\theta)$ and $P^1_{1}(r,\theta)$ stems from the Koopman modes $V_{1 1} = 1/2$ and $V_{{\langle 11\rangle} 1} = 1/4$ in the KMD, leading to the relation in the above equation.

\section{{Numerical Method}}
\label{sec:NA}

This section develops a numerical method to estimate the PFs for nonlinear systems. 
From Definitions~\ref{def_PFGP_NL} and \ref{def_PFGP_NL_High}, it is required to derive the Koopman eigenfunctions $\phi_j$. 
Although they could be derived in the examples of Section~\ref{sec:examples}, unfortunately, that is not the case for nonlinear systems in general.
A promising numerical method for estimating a set of Koopman eigenfunctions is the Extended DMD (EDMD) \cite{Williams2015, Klus2016}, where a set of observables or dictionary functions is utilized.
The EDMD is used for the PF computation in \cite{Takamichi2022, Netto2019}. 
The estimation accuracy provided by the EDMD depends on the availability of sufficiently rich time series data and appropriately chosen observables. 
See, e.g., \cite{Netto2021} for more detail.
Here, by using a classical technique in nonlinear oscillations \cite{Parker1989}, we propose an alternative {numerical} method for the PFs \emph{without} estimating the Koopman eigenfunctions. 
This method is also valid for estimating the mode-in-state GP.

\subsection{Proposed algorithm}
Before algorithm development, we provide a formulation of the numerical method guided by nonlinear dynamical system theory. 
The key idea is to start with the prolonged system in \cite{Mauroy2015} as follows:
%
\begin{equation}
\left\{
\begin{aligned}
\dot{\bf x}(t) &= {\bf F}({\bf x}(t)) \\
\dot{\bm \xi}(t)&={\rm D}{\bf F}({\bf x}(t)){\bm \xi}(t)
\end{aligned}
\right. 
\label{Prolonged}
\end{equation}
where the equation on $\bm \xi$ is called the variation equation \cite{Parker1989}. 
For \eqref{Prolonged}, according to \cite{Mauroy2015}, there exists $n$ Koopman eigenvalues $\lambda_j$ of multiplicity two.
The Koopman eigenfunctions belonging to $\lambda_j$ are as follows:
\begin{equation*}
\phi_j^{(1)}({\bf x}, \bm \xi) = \phi_j({\bf x}), \quad
\phi_j^{(2)}({\bf x}, \bm \xi) = \partial \phi_j({\bf x}) \bm \xi(t),
\end{equation*}
where $\phi_j({\bf x})$ is the Koopman eigenfunction for $\dot{\bf x}(t) = {\bf F}({\bf x}(t))$.
This relation holds for the higher-order Koopman eigenfunction $\phi_{\langle j_1\cdots j_n\rangle}({\bf x})$.
Here, we suppose that the time-evolution of $\xi_k$ can be expressed with the expansion on $\phi_j^{(2)}$ as follows:
\begin{equation}
\xi_k(t; {\bm \xi}^0)=\sum^n_{j=1}{\rm e}^{\lambda_jt}\underbrace{\phi_j^2({\bf x}^0, {\bm \xi}^0)V_{j \xi_k}}_{\hat{V}_{j \xi_k}}+{R_k(t)}, 
\label{eqn:KMD_variational}
\end{equation}
where {$R_k(t)$} stands for the residual term including the higher-order terms, $\lambda_j$ and $\phi_j$ {stand for} the $j$-th Koopman eigenvalue and eigenfunction of \eqref{Prolonged}, and $V_{j \xi_k}$ {is} the associated Koopman mode for the $k$-th variation $\xi_k$.
Here, the solution $\xi_k(t; {\bm \xi}^0)$ starting at ${\bm \xi}^0=[0,\ldots,0,\delta x_k,0,\ldots,0]^\top$ represents the variational dynamics of the $k$-th state $x_k$ after the initial variation $\delta x_k$ in the $k$-th state.
It corresponds to ${\rm d}x_k(t)$ in \eqref{dx} with $\delta x_\ell = 0~(\ell \neq k)$:
\begin{equation}
{\rm d}x_k(t; {\bf x}^0, {\delta{\bf x}})
= \sum^{n}_{j=1}{\rm e}^{\lambda_jt} {P_{j}^{k}}({\bf x}^0) \delta x_k + R^\prime_k(t),
\label{dxk_NA}
\end{equation}
where $R^\prime_k(t)$ stands for the residual term including the higher-order terms.

%
By comparing \eqref{eqn:KMD_variational} with \eqref{dxk_NA}, we have
\begin{equation}
\hat{V}_{j \xi_k}=P^k_j({\bf x}^0)\delta x_k.
\label{eqn:KEY}
\end{equation}
As mentioned below, $\hat{V}_{j \xi_k}$ can be computed through standard DMD for time series data of ${\bm \xi}$. 
Hence, with the information on $\delta x_k$, Equation~\eqref{eqn:KEY} makes it possible to directly estimate the PF $P^k_j({\bf x}^0)$ without estimating Koopman eigenfunctions. 
This idea is applicable to the numerical {estimation} of the mode-in-state GP $P^{k(\ell)}_j({\bf x}^0)$ by changing the initial variation $\delta x_k$ to $\delta x_\ell$ at the starting point $\bm \xi^0$. 
The higher-order variants $P^k_{\langle j_1\cdots j_n\rangle}({\bf x})$ and $P^{k (\ell)}_{\langle j_1\cdots j_n\rangle}({\bf x})$ can be estimated in the same manner as above.\footnote{It should be mentioned that estimating higher-order Koopman eigenvalues and eigenfunctions accurately is still challenging in the current existing DMD.} 

We now outline the algorithm for the above idea. 
As a working assumption, the model ${\bf F}({\bf x})$ of the target nonlinear system \eqref{Nonlinear_ODE} is known. 
The algorithm consists of three steps.
For the first step, the solution of the prolonged system \eqref{Prolonged} 
starting from the initial conditions ${\bf x}^0$ and ${\bm \xi^0}=\displaystyle [0,\ldots,0,\underbrace{\Delta}_{k{\rm -th}},0,\ldots,0]^\top$ is numerically pursued. 
{Here, $\Delta$ is a small real number, which can be set arbitrarily.}
Through uniform sampling $h$, time series data with a finite-integer length, 
\[
\left\{{\bf y}[t]:=[{\bf x}[t]^\top\,{\bm \xi}[t]^\top]^\top ~|~t=0,1,\ldots,{\rm (finite)} \right\},
\]
are obtained. 
For the second step, by using standard DMD (see, e.g., \cite{Kutz2016} and Chapter~7 of \cite{KoopmanBook}), the KMD of ${\bf y}[t]$ is estimated as follows:
\[
{\bf y}[t]\approx\sum^{\rm finite}_{j=1}\hat{\rho}_j^t{\bf \hat{V}}_j, \quad t=0,1,\ldots,{\rm (finite)},
\]
where $\hat{\rho}_j$ is the $j$-th DMD eigenvalue and ${\bf \hat{V}}_j$ is the associated DMD eigenvector.  
The DMD eigenvalues and eigenvectors provide approximations to the Koopman eigenvalues, $\lambda_{j} \approx \hat \lambda_j =\ln{\left(\hat{\rho}_{j}\right)} / {h}$, and (up to a constant $c\in\mathbb{R}$) to the Koopman modes, ${\bf V}_{j} \approx c\cdot{\bf \hat{V}}_{j}$. 
For the continuous time, the estimated KMD of $\bf y$ is written as follows:
\[
{\bf y}(t)
\approx \sum^{\rm finite}_{j=1} e^{\hat \lambda_j t} {\bf \hat{V}}_j.
\]
For the third step, from \eqref{eqn:KEY}, the mode-in-state PF $P^k_j({\bf x}^0)$ is approximately the following:
\[
P^k_j({\bf x}^0)\approx \frac{\hat{V}_{j \xi_k}}{{\Delta}}.
\]
where $\hat V_{j \xi_k}$ {is} an element of the estimated DMD eigenvector $\hat {\bf V}_j$ for the $k$-th variation $\xi_k$.
Here, by \eqref{eqn:KEY}, $V_{j\xi_k}$
depends on ${\bf x}_0$ and ${\bm \xi}_0$
Note that the index $j$ is chosen such that the estimated Koopman eigenvalue $\hat{\lambda}_j$ is close to that of a target mode of our analysis, for example, frequency and damping rate.
The algorithm procedure of the proposed method is summarized in Appendix~\ref{app_Alg}. 

Before moving on with a demonstration, it should be mentioned that the above algorithm {is not applied} to the numerical estimation of the state-in-mode GP $P^k_{i(j)}({\bf x}^0)$ (or $P^k_{ \langle j_1\cdots j_n\rangle (j)}({\bf x}^0)$). 
This is because it is required to extract the evolution of the modal variables $z_j$ and $z_{{\langle j_1\cdots j_n\rangle}}$, and this is not straightforward from the direct computation of the target system \eqref{Nonlinear_ODE}, namely, computation of the state's dynamics. 
For the state-in-mode GP, it is currently required to estimate the Koopman eigenfunctions {using} the EDMD.

\subsection{Demonstration} 

Now, we demonstrate the above algorithm for the two simple systems in Section~\ref{sec:examples}. 
That is done by computing solutions to a nonlinear system, from $N$ initial points with a MATLAB solver, where ${\bf x}^0_n$ denotes the $n$-th initial point ($n=1,\ldots,N$).
For example, let us take the case of mode-in-state GP $P^{1(2)}_1({\bf x}^0)$. 
The associated computational result is denoted by $\hat{P}^{1(2)}_1({\bf x}^0_n)$. 
To quantify the computational accuracy of the algorithm, we introduce the mean error $err(P^{1(2)}_1)$ as follows:
\begin{equation}
err(P^{1(2)}_1) := \frac{1}{N}\sum^N_{n=1}\left|P^{1(2)}_1({\bf x}^0_n)-\hat{P}^{1(2)}_1({\bf x}^0_n)\right|.
\label{Err}
\end{equation}
In this demonstration, the Prony-type DMD \cite{Susuki2015} is utilized, and an estimated Koopman eigenvalue $\hat \lambda_j$, which is close to an analytically derived value of the target Koopman eigenvalue is addressed. 
Note that, at some initial point, one can not estimate PFs and mode-in-state GP appropriately, since no $\hat{\lambda}_j$ close to the target $\lambda_j$ is obtained. 

\subsubsection{Nonlinear system \eqref{EX1_System}}

Figure~\ref{Fig:EX3} shows the computational result of mode-in-state GP ${\hat P}_{1}^{1(2)}({\bf x}^{0}_{n})$. 
In the figure, $x_1^{0}$ and $x_2^{0}$ are in the range $[-6,-6]$, the computation is done for $N=201 \times 201$ initial points ${\bf x}^0_n$ and $\Delta$ is set to $10^{-6}$. 
To estimate the mode-in-state GP, for each point, we generated the time series data $\{{\bf y}[0],\ldots, {\bf y}[5]\}$ through a uniform sampling $h = 0.3$, by solving \eqref{EX1_System} with a MATLAB solver. 
The setting was learned by trial and error.
We see that the computational result $\hat{P}^{1(2)}_1({\bf x}^0_n)$ is consistent with the true GP $P^{1(2)}_1({\bf x})=\frac{2+4\sqrt{2}}{7}x_2$. 
Here, we investigate the dependence of the computational result on the initial variation $\Delta$.
Table~\ref{table:Xi} shows the result of $err(P^{1(2)}_1)$, $err(P^{1(2)}_2)$, and $err(P^{1(2)}_{{{\langle 02 \rangle}}})$ under $\Delta = 10^{-6}, 10^{-3}, 1$, implying that in this demonstration, the setting of $\Delta$ does not change the computational result of {the} mode-in-state GP.
That can be explained by the linearity {of} the governing equation of $\bm \xi$. 
The variational equation on $\bm \xi$ is regarded as a time-variant linear system since that is introduced based on the Jacobian matrix {\rm D}{\bf F}.
Given that the system's state dynamics converge to stable EP or LC, the setting of $\Delta$ does not give a significant difference to the computational result of PFs and mode-in-state GP.
Table 
\ref{table:Xi} also shows the mode-in-state GP on higher-order mode $P^{1(2)}_{\langle 02 \rangle} ({\bf x})$ can also be estimated and the estimation is highly accurate. 
However, note that estimating the higher-order mode is generally challenging.

\begin{figure}[t] 
    \centering
   \includegraphics[width=0.6\textwidth]{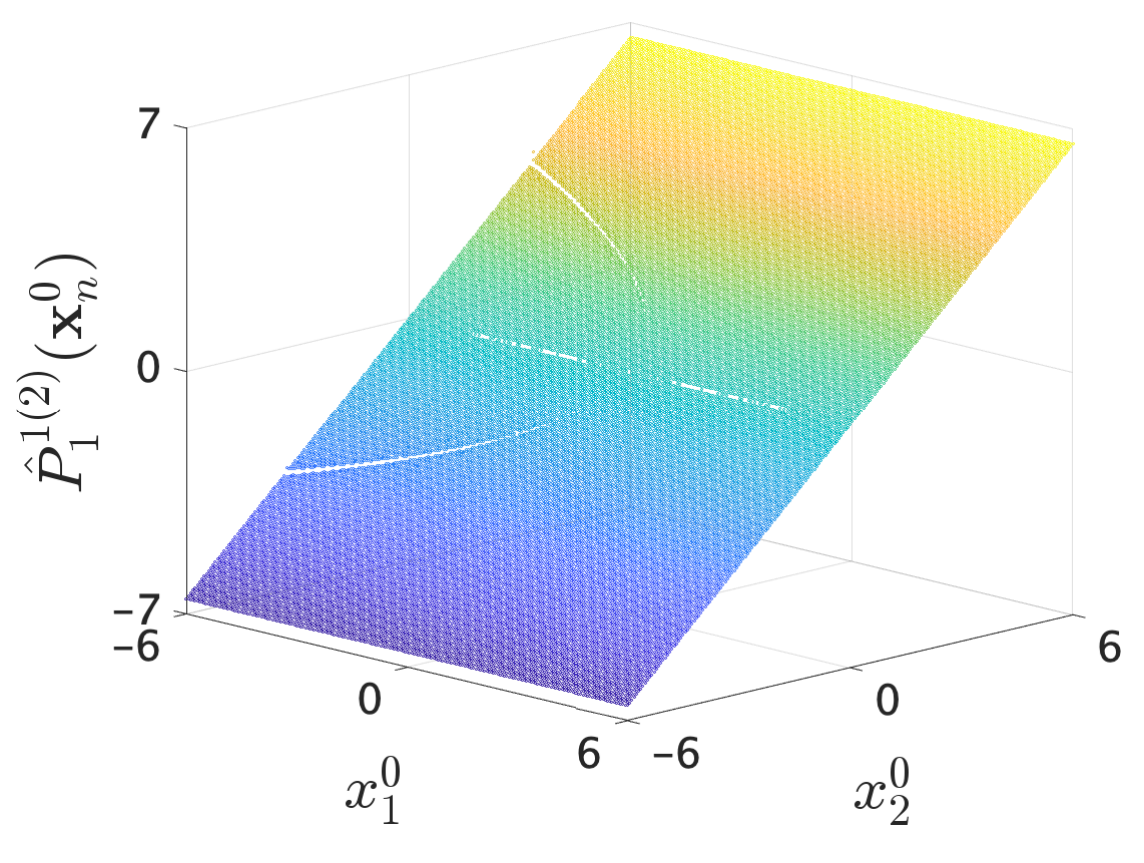}
   \caption{
   Computational result of the mode-in-state GP $\hat{P}_{1}^{1(2)}{({\bf x}^0_{n})}$ for the nonlinear system \eqref{EX1_System} with the stable equilibrium point. 
   The initial change $\Delta$ is set to $10^{-6}$.
   The result is consistent with the analytical form $\frac{2+4\sqrt{2}}{7}x_2$. 
   }
   \label{Fig:EX3}
\end{figure}

\begin{table}[t]
  \caption{Results of $err(P^{1(2)}_1)$, $err(P^{1(2)}_2)$ and $err(P^{1(2)}_{{\langle 02 \rangle}})$ for the nonlinear system \eqref{EX1_System} with the stable equilibrium point under $\Delta = 10^{-6}, 10^{-3}, 1$.}
  \label{table:Xi}
  \centering
  \begin{tabular}{|c|c|c|c|}
    \hline \noalign{\vskip 0.05mm}
                                     & $\Delta = 10^{-6}$          &  $\Delta = 10^{-3}$        & $\Delta = 1$ \\
    \hline 
    $err(P^{1(2)}_1)$    & $0.0012$  & $0.0007$ &  $0.0008$  \\
    \hline
    $err(P^{1(2)}_2)$ & $0.0016$  & $0.0009$ &  $0.0010$  \\
    \hline
    $err(P^{1(2)}_{{\langle 02\rangle}})$  & $0.0002$  & $0.0002$ & $0.0003$  \\
    \hline
  \end{tabular}
 \end{table}

\subsubsection{Nonlinear system \eqref{EX2_System}}

 \begin{figure*}[h]
  \begin{minipage}[b]{0.45\linewidth}
    \centering
    \includegraphics[width=\hsize]{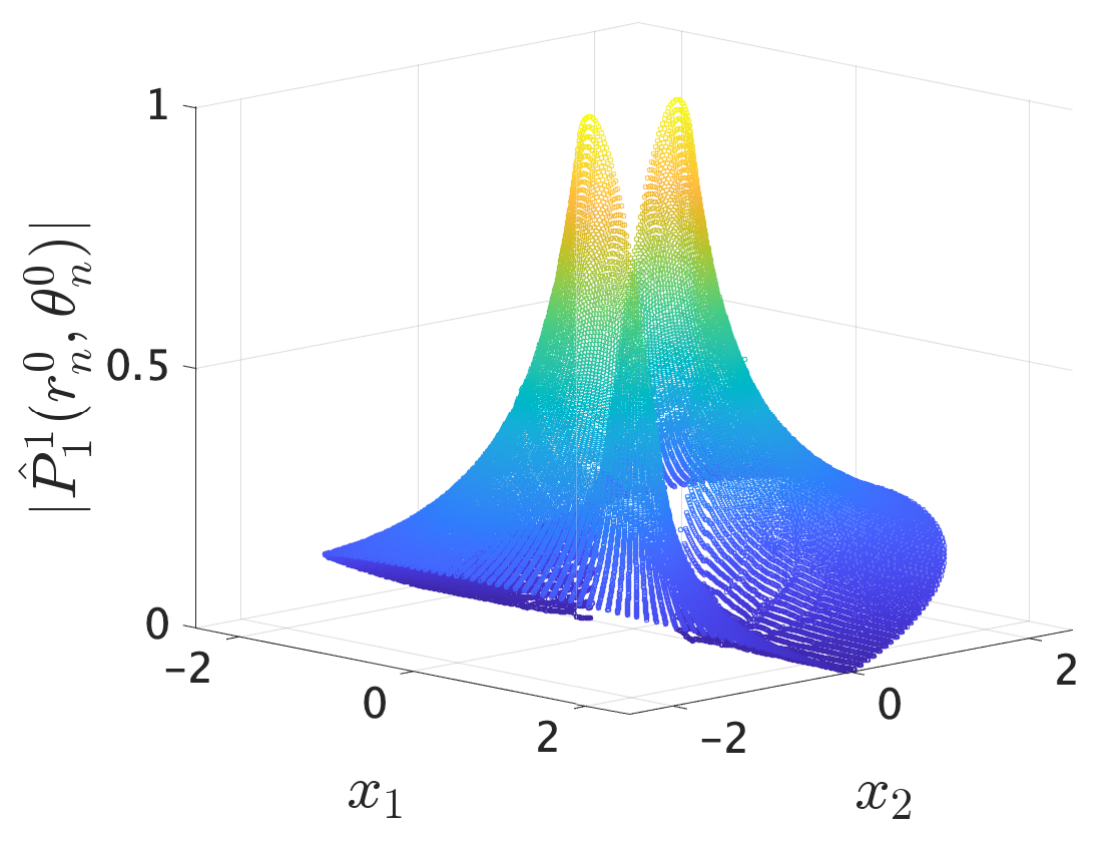}
    \subcaption{Result of $\hat{P}_{1}^{1}(r^{0}_{n},\theta^{0}_{n})$}
    \label{Fig:EX4_Alg_Phi1}
  \end{minipage}
  \begin{minipage}[b]{0.45\linewidth}
    \centering
    \includegraphics[width=\hsize]{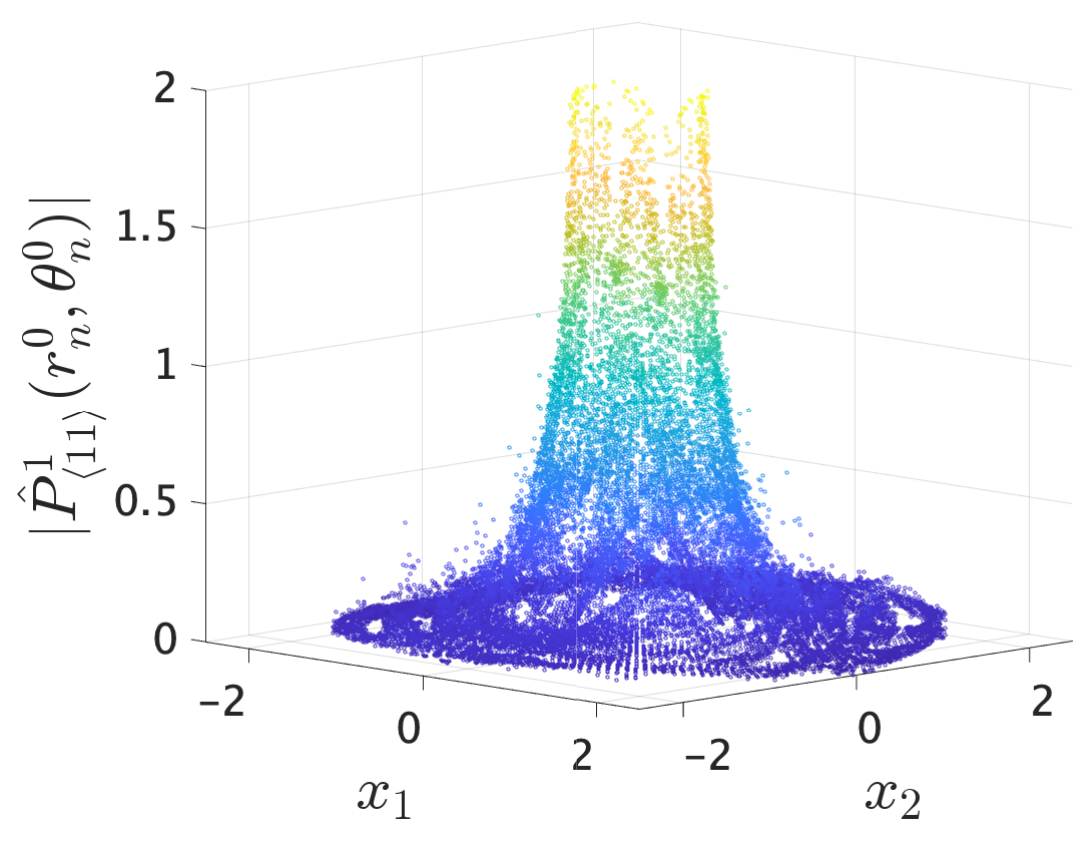}
    \subcaption{Result of $\hat{P}^1_{{\langle 11\rangle}}(r^{0}_{n}, \theta^{0}_{n})$}
     \label{Fig:EX4_Alg_Phi1-1}
  \end{minipage}
   \begin{minipage}[b]{0.45\linewidth}
    \centering
    \includegraphics[width=\hsize]{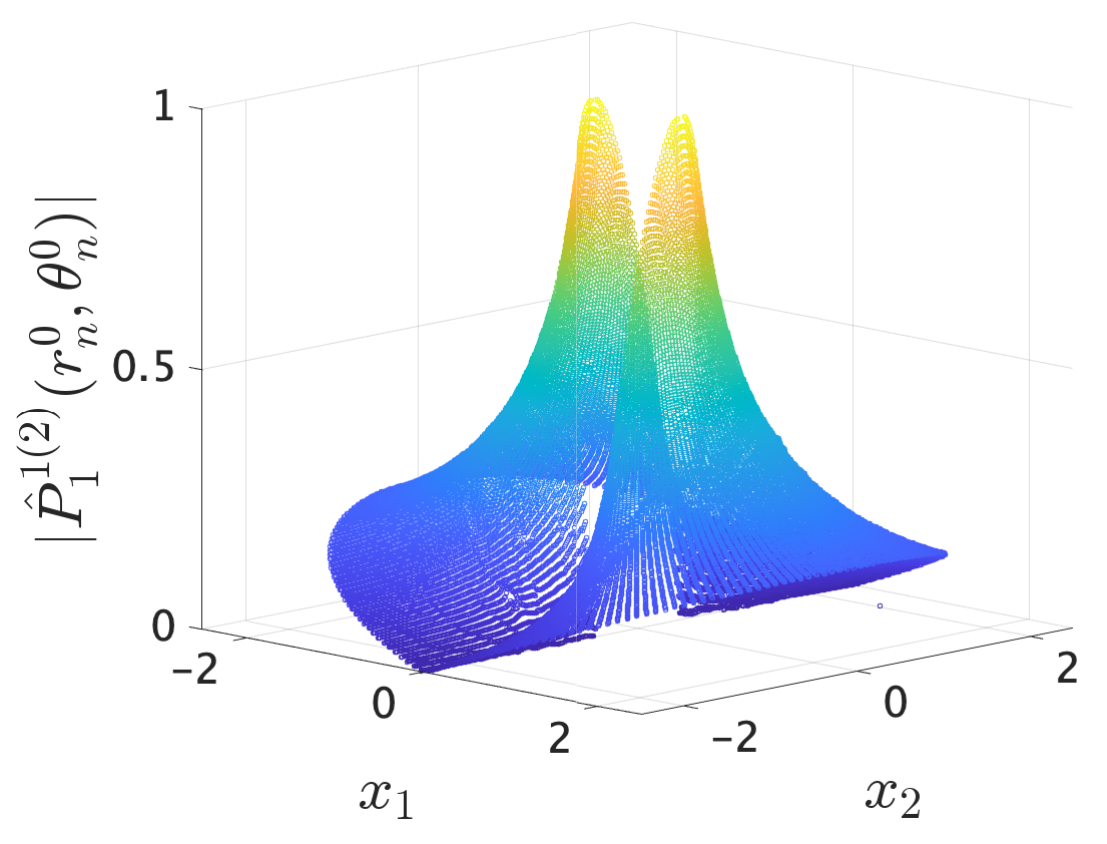}
    \subcaption{Result of $\hat{P}^{1(2)}_{1}(r^{0}_{n},\theta^{0}_{n})$}
     \label{Fig:EX4_Alg_Phi1_GP}
  \end{minipage}
  \caption{
  Computational results of the PFs $\hat{P}_{1}^{1}(r^{0}_{n},\theta^{0}_{n})$ and $\hat{P}^1_{{\langle 11\rangle}}(r^{0}_{n},\theta^{0}_{n})$, and the mode-in-state GP $\hat{P}_{1}^{1(2)}(r^{0}_{n},\theta^{0}_{n})$ for the nonlinear system \eqref{EX2_System} with the stable limit cycle. 
  The result of $\hat{P}^1_{{\langle 11\rangle}}(r^{0}_{n},\theta^{0}_{n})$ is depicted in $|\hat{P}^1_{{\langle 11\rangle}}(r^{0}_{n},\theta^{0}_{n})| \leq 2$ because many of the computational results were in this range. 
  For the mode $z_1$ of the rotational component, the numerical results $\hat{P}_{1}^{1}(r^{0}_{n},\theta^{0}_{n})$ and $\hat{P}_{1}^{1(2)}(r^{0}_{n},\theta^{0}_{n})$ are consistent with the true ${P}_{1}^{1}(r,\theta)$ and ${P}_{1}^{1(2)}(r,\theta)$ in Fig.~\ref{Figs_EX2}.
  {For the mode $z_{{\langle 11\rangle}}$ of complex {component},  
  although at some points no $\hat \lambda$ close to the target $\lambda_j$ is obtained, the numerical result is similar to the true ${P}^1_{{\langle 11\rangle}}(r,\theta)$ in Fig.~\ref{Figs_EX2}.} 
  }
  \label{Fig:EX4_Alg_PF}
\end{figure*}

Figure~\ref{Fig:EX4_Alg_Phi1} shows the computational result of the PF ${\hat P}_{1}^{1}(r^{0}_{n},\theta^{0}_{n})$. 
In the figure, $r^{0}$ is in the range $[0.5, 2.5]$, $\theta^{0}$ is in the range $[-\pi, \pi)$, $\Delta$ is set to $10^{-6}$, and the computation is done for $N=201 \times 201$ initial points $(r^0_n,\theta^0_n)$.  
To estimate the PFs, for each point, we generated the time series data $\{{\bf y}[0], \ldots, {\bf y}[99]\}$ through a uniform sampling $h = 0.1$, by solving \eqref{EX2_System} with a {MATLAB} solver. 
We see that the computational result $\hat{P}_{1}^{1}(r^0_n,\theta^0_n)$ is consistent with the true PF ${P^1_{1}}(r,\theta) = -{\rm i}\frac{\sin \theta}{2r}{\rm e}^{{\rm i}\theta} $ in Fig.~\ref{Fig:EX2_PF}.
In this experiment, $err(P^{1}_{1}) = 0.0017$ is obtained, implying the correctness of the computational result of the PFs.

Figure~\ref{Fig:EX4_Alg_Phi1-1} shows the computational result of the PF $\hat{P}_{{\langle 11\rangle}}^{1}(r^{0}_{n},\theta^{0}_{n})$.
The computational setting is the same as for ${\hat P}_{1}^{1}(r^{0}_{n},\theta^{0}_{n})$.
At some points, our computation of $\hat{P}_{{\langle 11\rangle}}^{1}(r^0_n,\theta^0_n)$ was not performed because any estimated Koopman eigenvalue $\hat \lambda_j$ did not have close values to ${\rm i} - 2$ which is the analytical value of the Koopman eigenvalue associated with $\phi_{1}(r,\theta)^1\phi_{2}(r,\theta)^1$.
Here, $\hat{P}_{{\langle 11\rangle}}^{1}(r^0_n,\theta^0_n)$ is computed at $27588$ points out of $40401$ points.
From Fig.~\ref{Fig:EX4_Alg_Phi1-1}, as for the point where the computation was possible, the result of $\hat{P}_{{\langle 11\rangle}}^{1}(r_n^0,\theta_n^0)$ shows the same characteristic of $P^1_{{\langle 11\rangle}}(r,\theta)$ in Fig.~\ref{Fig:EX2_PF_Complex}. 
Here, $err(P^{1}_{{\langle 11\rangle}}) = 0.1628$ is obtained, where the mean error is mainly due to the difficulty of {the} DMD computation inside the LC. 
When we set $r$ in the range $[1, 2.5]$ without changing other settings, $err(P^{1}_{{\langle 11\rangle}}) = 0.0397$ is obtained, which enhances the correctness of the computational result.

Figure~\ref{Fig:EX4_Alg_Phi1_GP} shows the computational result of the mode-in-state GP ${\hat P}_{1}^{1(2)}(r^{0}_{n},\theta^{0}_{n})$. 
We see that the computational result $\hat{P}_{1}^{1(2)}(r^0_n,\theta^0_n)$ is consistent with the true GP $P^{1(2)}_1(r,\theta) = {\rm i}\frac{\cos \theta}{2r}{\rm e}^{{\rm i}\theta}$ in Fig.~\ref{Fig:EX2_GP_MinS}.
In this experiment, $err(P^{1(2)}_{1}) = 0.0017$ is obtained, implying the correctness of the computational result of the mode-in-state GP.

\section{Conclusion}
\label{sec:outro}

This paper proposed novel {definitions} of PFs and GPs for nonlinear systems with stable EPs and LCs. 
To this end, two main ideas were introduced. 
First, we reinterpreted the classical concept of PFs and GPs in LTI systems from the viewpoint of variational dynamics. 
Second, we utilized the Koopman operator framework for nonlinear systems, from which we particularly leveraged the Koopman eigenfunctions and KMD. 
Then, by combining these, we introduced novel definitions of PFs and GPs for nonlinear systems, in which the classical theory of PFs and GPs is regarded as a special case. 
We also developed the numerical method for the proposed PFs without estimating the Koopman eigenfunctions, which is still a challenging issue in data-driven science. 
The proposed algorithm is also applied to mode-in-state GP.
The algorithm is based on the traditional variational equation and can be implemented with many variants of DMD. 
Its effectiveness was demonstrated for the two simple models of the nonlinear systems. 

Several follow-up studies are possible. 
First, the theory developed in this paper is also applied to systems with quasi-periodic attractors without any significant modification: see \cite{Mezic2020} for the Koopman operator theory of quasi-periodic attractors.
Second, developing the numerical method to estimate the state-in-mode GP is promising and left for future work.
Finally, to show that the theory is of technical importance, leveraging the proposed theory into power system problems is an exciting direction. 
This is because of the relevance of participation factors to the analysis and control of electric power grids.
The work in \cite{Takamichi2022} represents a first step in that direction, where we applied participation factors to a nonlinear model of interconnected AC/DC grids and showed its capability of revealing the dynamic interplay between state variables and modes.

\appendix

\section{The modulated Segal-Bargmann space}
\label{app:Segal}

The modulated Segal-Bargmann space is a reproducing kernel Hilbert space, which was originally introduced in \cite{Mezic2020}. 
We here briefly introduce the space based on \cite{Susuki2021}. 
Using the Koopman eigenfunctions, we define the conjugacy in the entire domain $\mathbb{B}({\bf x}^\ast)$ as follows:
\[
{\bf h}: {\mathbb B}({\bf x}^\ast) \rightarrow {\mathbb C}^n, \quad {\bf h}({\bf x}) = (\phi_1({\bf x}), \ldots, \phi_n({\bf x}))^\top,
\]
such that
\[
{\bf h} \circ {\bf S}^t= \exp({{\bf D} t}) \circ {\bf h},
\]
where ${\bf D}:= {\rm diag}(\lambda_1, {\ldots}, \lambda_n)$ is a diagonal matrix consisting of the Koopman principal eigenvalues. 
Here, following pages~2023--2024 in \cite{Susuki2021}, we introduce the modulated Segal-Bargmann space $\mathscr{S}_{\bf h}$ of functions that are entire in ${\bf h}$, i.e., of the form $f=g\circ{\bf h}$, where $g: \mathbb{C}^n\to\mathbb{C}$ is an entire function. 
This space $\mathscr{S}_{\bf h}$ is endowed with the norm and inner product as follows: 
\[
\begin{aligned}
\|{f}\|_{\mathscr{S}_h} 
&= \|{g} \circ {\bf h}\|_{\mathscr{S}_h}
= \frac{1}{\pi^n} \int_{{\mathbb C}^n} |{g}({\bf z})|^2 {\rm e}^{-|{\bf z}^2|}{\rm d}{\bf z} < \infty, \\
\langle{{f}_1, {f}_2}\rangle_{\mathscr{S}_h}  
&= \frac{1}{\pi^n} \int_{{\mathbb C}^n} {g}_1({\bf z}) \overline{{g}_2({\bf z})} {\rm e}^{-|{\bf z}^2|}{\rm d}{\bf z}, 
\quad {f}_i = {g}_i \circ {\bf h}.
\end{aligned}
\]
By working in this space, it is shown (see the same pages in \cite{Susuki2021}) that the series based on the Koopman eigenfunctions (see Lemma~\ref{lemma:KEF_EP}), given by
\[
\begin{aligned}
\mathcal{U}^t f({\bf x})
=& \sum_{j_1,\ldots, j_n \in {\mathbb N}} {\rm e}^{(j_1 \lambda_1 + \cdots +j_n \lambda_n)t} \cdot 
\nonumber\\
& \cdot \phi_1({\bf x})^{j_1}\cdots\phi_n({\bf x})^{j_n}{V}_{\langle j_1\cdots j_n\rangle},~t\geq 0,
\end{aligned}
\]
is pointwise convergent. 
Here, the Koopman mode $V_{\langle j_1\cdots j_n\rangle}$ corresponds to $\langle{f({\bf x}),\phi_1({\bf x})^{j_1}\cdots\phi_n({\bf x})^{j_n}}\rangle_{\mathscr{S}_h}$. 
Moreover, the Koopman eigenfunctions form a complete orthonormal basis, and the spectrum is totally disconnected (no continuous part of the spectrum exists in this space).

\section{Proof of Theorem~\ref{DEF_KMD_PF}}
\label{app:DEF_KMD_PF}

First, we derive \eqref{dx}. 
Under an infinitesimal change ${\delta{\bf x}}$ 
such that ${\bf x}^0+{\delta{\bf x}}\in\mathbb{B}({\bf x}^\ast)$,\footnote{It is possible to choose such $\delta{\bf x}$ because $\mathbb{B}({\bf x}^\ast)$ is open \cite{Chiang1988}.} 
the variation ${\rm d}x_k(t; {\bf x}^0, {\delta{\bf x}})$ in \eqref{x{j_n}onlinear} is represented as
\begin{equation}
{\rm d}x_k(t; {\bf x}^0, {\delta{\bf x}}) 
=
\sum^{n}_{\ell=1} \left.\frac{\partial}{\partial x_\ell}\bar{x}_k(t; {\bf x})\right|_{{\bf x}={\bf x}^0} \delta x_\ell + {\cal O}(\|\delta{\bf x}\|^2), 
\label{dx_app}
\end{equation}
where the partial derivative on the right-hand side is the derivative of $\bar{x}_k(t; {\bf x}^0)$ by the $\ell$-th element $x^0_\ell$ of the initial state ${\bf x}^0$ and ${\cal O}$ is the Landau's symbol. 
The existence of the partial derivative (that is, differentiability with respect to the initial state) is from the existence of smooth Koopman eigenfunctions in Lemma~\ref{lemma:KEF_EP} and the KMD \eqref{KMD_X} in Lemma~\ref{lemma:KMD_EP}. 
Since both ${\bf x}^0$ and ${\bf x}^0+{\delta{\bf x}}$ are in $\mathbb{B}({\bf x}^\ast)$, the variation ${\rm d}x_k(t; {\bf x}^0, {\delta{\bf x}})$ converges to $0$ as $t\to+\infty$. 
By again using the KMD \eqref{KMD_X}, we have 
\begin{align}
{\rm d}x_k(t; {\bf x}^0, {\delta{\bf x}})
&\approx \left.\frac{\partial}{\partial x_k}\bar{x}_k(t; {\bf x})\right|_{{\bf x}={\bf x}^0}\delta x_{k}
+ \sum_{\substack{\ell = 1\\\ell \neq k}}^{n}
\left.\frac{\partial}{\partial x_\ell}\bar{x}_k(t; {\bf x})\right|_{{\bf x}={\bf x}^0} \delta x_\ell 
\notag \\
&= \sum_{j=1}^{n} {\rm e}^{\lambda_j t} \frac{\partial\phi_j}{\partial x_k}({\bf x}^0) V_{{jk}} \delta x_k 
+ \sum_{\substack{\ell = 1\\\ell \neq k}}^{n} 
\sum_{j=1}^{n} {\rm e}^{\lambda_j t} \frac{\partial\phi_j}{\partial x_\ell}({\bf x}^0) V_{{jk}} \delta x_\ell
\notag \\
&+ \sum^\infty_{\substack{{j_1},\ldots,{j_n}\in\mathbb{N}_0 \\ {j_1}+\cdots+{j_n}>1}} 
{\rm e}^{({j_1}\lambda_1+\cdots+{j_n}\lambda_n)t}   
\frac{\partial \phi_{\langle j_1\cdots j_n\rangle}}{\partial x_k}({\bf x}^0)V_{\langle j_1\cdots j_n\rangle k} \delta x_k  \nonumber  \\
&+ \sum_{\substack{\ell = 1\\\ell \neq k}}^{n} 
\sum^\infty_{\substack{{j_1},\ldots,{j_n}\in\mathbb{N}_0 \\ {j_1}+\cdots+{j_n}>1}}
{\rm e}^{({j_1}\lambda_1+\cdots+{j_n}\lambda_n)t}
\frac{\partial \phi_{\langle j_1\cdots j_n\rangle}}{\partial x_\ell}({\bf x}^0) 
V_{\langle j_1\cdots j_n\rangle k} \delta x_\ell. 
\label{dxk_pf_EP}
\end{align}
In \eqref{dxk_pf_EP}, the notation $\approx$ is used because ${\cal O}(\|\delta{\bf x}\|^2)$ in \eqref{dx_app} is omitted. 
By combining \eqref{dxk_pf_EP} with Definitions~\ref{def_PFGP_NL} and \ref{def_PFGP_NL_High}, the derivation of \eqref{dx} is completed. 

Next, we derive \eqref{dz}. 
Under the infinitesimal change $\delta {\bf x}$ such that ${\bf x}^0+{\delta{\bf x}}\in\mathbb{B}({\bf x}^\ast)$, by utilizing \eqref{x{j_n}onlinear}, the variation ${\rm d}z_j(t; {\bf x}^0, {\delta{\bf x}})$ in \eqref{z{j_n}onlinear} is represented as 
\begin{align}
{\rm d}z_j(t) &\approx
\sum_{k=1}^{n} \left. \frac{\partial }{\partial x_k}\bar z_j(t; {\bf x})\right|_{{\bf x}={\bf x}^0}
\delta x_k \nonumber \\
&= 
\sum_{k=1}^{n} \left. \frac{\partial }{\partial x_k}\bar z_j(t; {\bf x})\right|_{{\bf x}={\bf x}^0} 
{\rm d} x_k(0; {\bf x}^0, {\delta {\bf x}}) 
\label{dz_app}
\end{align}
where the partial derivative on the right-hand side is the derivative of $\bar z_j(t; {{\bf x}^0})$ by the $k$-th element $x_k^0$ of the initial state ${\bf x}^0$.
The existence of the partial derivative (that is, differentiability with respect to the initial state) is from the existence of smooth Koopman eigenfunctions in Lemma~\ref{lemma:KEF_EP}.
Since the Koopman eigenvalues $\lambda_1,\ldots,\lambda_n$ are negative, the modal evolution $\bar{z}_j(t; {\bf x}^0) = {\rm e}^{\lambda_j t}\phi_j({\bf x}^0)$ converge to $0$ as $t \rightarrow +\infty$.
Hence, the variation ${\rm d}z_j(t)$ also converges to $0$ as $t \rightarrow +\infty$.
Here, by using the KMD \eqref{KMD_X}, we have 
\begin{align}
{\rm d}x_k(0; {\bf x}^0, {\delta {\bf x}}) 
&= x_k(0; {\bf x}^0+{\delta}{\bf x}) - x_k(0; {\bf x}^0) \nonumber \\
&= \sum_{j=1}^n V_{{jk}} \{\phi_j({\bf x}^0+{\delta}{\bf x}) - \phi_j({\bf x}^0)\} \nonumber \\
&+ \sum^\infty_{\substack{{j_1},\ldots,{j_n}\in\mathbb{N}_0 \\ {j_1}+\cdots+{j_n}>1}} \makebox[-1em]{}
V_{\langle j_1\cdots j_n\rangle k} 
\{ {\phi_{\langle j_1\cdots j_n\rangle}} ({\bf x}^0+{\delta}{\bf x}) - 
{\phi_{\langle j_1\cdots j_n\rangle}} ({\bf x}^0)\} \nonumber \\
&= \sum_{j=1}^n V_{{jk}} \delta z_j^0 
+  \sum^\infty_{\substack{{j_1},\ldots,{j_n}\in\mathbb{N}_0 \\ {j_1}+\cdots+{j_n}>1}} V_{\langle j_1\cdots j_n\rangle k} \delta z_{\langle j_1\cdots j_n\rangle}^0,
\label{Eq:dx0}
\end{align}
where $\delta z_j^0$ and $\delta z_{\langle j_1\cdots j_n\rangle}^0$ represent the initial change{s} to the modes associated with the initial change $\delta {\bf x}$, defined in Theorem \ref{lemma:KEF_EP}. 
By substituting \eqref{Eq:dx0} {into \eqref{dz_app}} and using $\bar{z}_{j}(t; {\bf x}^0) = {\rm e}^{\lambda_jt} \phi_{j}({\bf x}^0)$, ${\rm d}z_j(t)$ is expressed as follows:
\begin{align}
{\rm d}z_j(t; {\bf x}^0, {\delta{\bf x}}) 
& \approx \sum_{k=1}^{n} {\rm e}^{\lambda_j t} \frac{\partial \phi_j}{\partial x_k} ({\bf x}^0) V _{{jk}} 
\delta z_j^0 \notag \\
& + \sum_{\substack{i = 1\\i \neq j}}^{n} \sum_{k=1}^{n} {\rm e}^{\lambda_j t} \frac{\partial \phi_j}{\partial x_k}({\bf x}^0) V_{{ik}} 
\delta z_i^0 \notag \\
& + \sum^\infty_{\substack{{j_1},\ldots,{j_n}\in\mathbb{N}_0 \\ {j_1}+\cdots+{j_n}>1}} 
 \sum_{k=1}^{n}  
{\rm e}^{\lambda_j t} \frac{\partial \phi_j}{\partial x_k}({\bf x}^0) V_{\langle j_1\cdots j_n\rangle k} 
\delta z_{\langle j_1\cdots j_n\rangle}^0.
\label{dzj_pf_EP}
\end{align}
Note that the second and third terms on the right-hand side of \eqref{dzj_pf_EP} do not appear for the LTI system. 
By combining \eqref{dzj_pf_EP} with Definitions~\ref{def_PFGP_NL} and \ref{def_PFGP_NL_High}, both the derivation of \eqref{dz} and the proof of the theorem are completed.
%

\section{Proof of Theorem~\ref{Case_LC}}
\label{app_LC}

It is stated in Lemmas~\ref{lemma:KE_LC} and \ref{lemma:mode_LC} that there exist Koopman eigenvalues and eigenfunctions for the nonlinear system \eqref{Nonlinear_ODE} with the stable LC $\Gamma$. 
From \cite{Mezic2020}, by taking $\mathscr{F}$ as the averaging kernel Hilbert space \cite{Mezic2020} defined on $\mathbb{B}(\Gamma)$,  for all ${\bf x}^0 \in \mathbb{B}(\Gamma)$ and $t \geq 0$, the flow of the system is represented as follows: 
\begin{align}
{\bf {\bar x}} (t) 
&= {\bf x}^\ast + \sum_{j=1}^{n} {\rm e}^{\lambda_j t} \phi_j ({\bf x}^0) {\bf V}_{j}
+  
\sum_{\mathcal{I}}
{\rm e}^{(  j_1 \lambda_1 + \cdots+j_n\lambda_{n} )t}
\phi_1^{j_1}({\bf x}^0) \cdots \phi_n^{j_n}({\bf x}^0)
{\bf V}_{\langle j_1\cdots j_n\rangle},
\label{eqn:KMD_LC}
\end{align}
where $\mathcal{I}=\mathcal{I}_+\cup \mathcal{I}_{-}$ are defined in Lemma~\ref{lemma:KE_LC}, and ${\bf V}_{j}$ and ${\bf V}_{\langle j_1\cdots j_n\rangle}$ are the Koopman modes determined with the inner product equipped on $\mathscr{F}$. 
The constant $x^\ast_k$ of ${\bf x}^\ast$ is the time average of $x_k(t)$ and corresponds to the inner product of a constant function, which is a Koopman eigenfunction of eigenvalue $0$, and $f({\bf x})=x_k$. 
See \cite{Susuki2021} for a rigor convergence proof of \eqref{eqn:KMD_LC}. 
In the same manner as {Appendix~}\ref{app:DEF_KMD_PF}, the variations ${\rm d}x_k(t; {\bf x}^0, \delta{\bf x})$ in \eqref{x{j_n}onlinear} and ${\rm d}z_j(t; {\bf x}^0,\delta{\bf x})$ in \eqref{z{j_n}onlinear} are represented as follows:
\begin{align}
{\rm d}x_k(t; {\bf x}^0, {\delta{\bf x}}) 
&\approx \left.\frac{\partial}{\partial x_k}\bar{x}_k(t; {\bf x})\right|_{{\bf x}={\bf x}^0}\delta x_{k}
+ \sum_{\substack{\ell = 1\\\ell \neq k}}^{n}
\left.\frac{\partial}{\partial x_\ell}\bar{x}_k(t; {\bf x})\right|_{{\bf x}={\bf x}^0} \delta x_\ell 
\notag \\
&= \sum_{j=1}^{n} {\rm e}^{\lambda_j t} \frac{\partial \phi_j}{\partial x_k}({\bf x}^0) V_{j k} \delta x_k 
+ \sum_{\substack{\ell = 1\\\ell \neq k}}^{n} 
\sum_{j=1}^{n} {\rm e}^{\lambda_j t} \frac{\partial\phi_j}{\partial x_\ell}({\bf x}^0) V_{{jk}} \delta x_\ell
\notag \\
&+ 
\sum_{\mathcal{I}} 
{\rm e}^{(j_1 \lambda_1 +\cdots+j_n\lambda_{n})t}
\frac{\partial \phi_{\langle j_1\cdots j_n\rangle}}{\partial x_k}({\bf x}^0)V_{\langle j_1\cdots j_n\rangle k} \delta x_k  \nonumber  \\
&+ \sum_{\substack{\ell = 1\\\ell \neq k}}^{n} 
\sum_{\mathcal{I}} 
{\rm e}^{(j_1 \lambda_1 +\cdots+j_n\lambda_{n})t}
\frac{\partial \phi_{\langle j_1\cdots j_n\rangle}}{\partial x_\ell}({\bf x}^0)V_{\langle j_1\cdots j_n\rangle k} \delta x_\ell,
\label{dxk_pf_LC}
\end{align}
and
\begin{align}
{\rm d}z_j(t; {\bf x}^0, {\delta{\bf x}}) 
&\approx  
\sum_{k=1}^{n} {\rm e}^{\lambda_j t} \frac{\partial \phi_j}{\partial x_k} ({\bf x}^0) V_{{jk}} \delta z_j^0
\notag \\
&  + \sum_{\substack{i = 1\\i \neq j}}^{n} \sum_{k=1}^{n} {\rm e}^{\lambda_j t} \frac{\partial \phi_j}{\partial x_k}({\bf x}^0) V_{{ik}} \delta z_i^0
\notag \\
&  + \sum_{\mathcal{I}} \sum_{k=1}^{n}  
{\rm e}^{\lambda_j t} \frac{\partial \phi_j}{\partial x_k}({\bf x}^0) V_{\langle j_1\cdots j_n\rangle k} \delta z_{\langle j_1\cdots j_n\rangle}^0,
\label{dzj_pf_LC}
\end{align}
where $\delta z_i^0$, $\delta z_j^0$, and $\delta z_{\langle j_1\cdots j_n\rangle}^0$ represent the initial change{s} to the modes associated with the initial change $\delta {\bf x}$ defined in Theorem~\ref{Case_LC}.  
By combining \eqref{dxk_pf_LC} and \eqref{dzj_pf_LC} with the PFs and GPs in Definitions~\ref{def_PFGP_NL} and \ref{def_PFGP_NL_High}, respectively, the proof of Theorem~\ref{Case_LC} is completed.

\section{Procedure of Proposed {Algorithm}}  \label{app_Alg}
The algorithm of the {numerical} method in Section~\ref{sec:NA} is summarized in Algorithm~\ref{alg1}.
In this paper, we utilize the Prony-type DMD \cite{Susuki2015}.
By this algorithm, the PF $P_j^k({\bf x}^{{0}})$ ({and} the mode-in-state GP $P_j^{k(\ell)}{({\bf x}^0})$) at ${\bf x}^0 = {\bf x}[0]$ {are} estimated.
Hence, to compute $P_j^k({\bf x}^{{0}})$ globally in the state space, we use this algorithm for a set of initial states on a predefined grid like the demonstration in Section~\ref{sec:NA}.

\begin{figure}[t]
\begin{algorithm}[H]
    \caption{Computational algorithm to estimate PFs}
    \label{alg1}
    \small
    \textbf{Input: } 
    Time series data ${\bf Y} = [{\bf y}[0], \ldots, {\bf y}[2N-1]]$ that start from ${\bf y}[0]:=[{\bf x}[0]^\top\,{\bm \xi}[0]^\top]^\top$, with sampling period $h$, and initial variation $\Delta$.
    \begin{enumerate}
	\item Form the vector Hankel matrix $\bf H$ and vector $\bf b$ as follows:
	   \begin{align}
   	 {\bf H} &:= 
    	\begin{pmatrix}
   	 {\bf y}[0] & {\bf y}[1] & \cdots & {\bf y}[N-1] \\
   	 {\bf y}[1] &  {\bf y}[2] & \cdots  & {\bf y}[N]\\
   	 \vdots & \vdots & \ddots & \vdots \\
   	 {\bf y}[N-1] & {\bf y}[N] & \cdots & {\bf y}[2N-2]
 	   \end{pmatrix}, \notag  \\
  	  \quad
   	 {\bf b} &:= -
  	  \begin{pmatrix}
  	  {\bf y}[N]  \\
  	  {\bf y}[N+1]  \\
  	  \vdots \\
 	   {\bf y}[2N-1]
 	   \end{pmatrix}. \notag 
 	   \end{align}  
	    \item Calculate $p_k$ as follows:
	  \begin{equation}
    		{\bf p} = [p_0,\ldots,p_{N-1}]^\top = ({\bf H}^\top {\bf H})^{-1} {\bf H}^\top {\bf b}.   \notag
	  \end{equation}
	\item Form the companion matrix given as 
	  \begin{equation}
   {\bf C}_N := 
    \begin{pmatrix}
    0 & 0 & \cdots & 0 & -p_0  \\
    1 & 0 & \cdots & 0 & -p_1  \\
    0 & 1 & \cdots &0 &  -p_2  \\
    \vdots & \vdots & \ddots & \vdots & \vdots \\
    0 & 0 & \cdots & 1 &  -p_{N-1}  \\
    \end{pmatrix}, \notag
  \end{equation}
  and calculate the eigenvalues of ${\bf C}_N$ which are the DMD eigenvalues $\hat{\rho}_j$.
	\item Form the Vandermonde matrix given as: 
	 \begin{equation}
 	  {\bf T} := 
  	  \begin{pmatrix}
  	  1 & \hat{\rho}_1 & \hat{\rho}_1^2 & \cdots & \hat{\rho}_1^{N-1}  \\
   	 1 & \hat{\rho}_2 & \hat{\rho}_2^2 & \cdots & \hat{\rho}_2^{N-1}  \\
	    \vdots & \vdots & \vdots & \ddots & \vdots \\
 	   1 & \hat{\rho}_{N} & \hat{\rho}_N^2 & \cdots & \hat{\rho}_N^{N-1}  \\
  	  \end{pmatrix}, \notag
  	  \end{equation}
	and calculate $\hat{\bf V}_j$ as follows:
	\begin{equation}
    		\hat{\bf V} = [\hat{\bf V}_1, \ldots, \hat{\bf V}_N] = [{\bf y}[0], \ldots, {\bf y}[N-1]] {\bf T}^{-1}.  \notag
 	 \end{equation}
	\item Calculate the estimated Koopman eigenvalue ${\hat{\lambda}_j} = {\rm ln}({\hat {\rho}_j}) / h$ and the PF ${\hat P}_{j}^{k}{({\bf x}^0)} = \hat{V}_{j \xi_k} / {\Delta}$.
    \end{enumerate}
    \textbf{Output: } $\hat{\lambda}_j$ (estimated Koopman eigenvalue), $\hat{\bf V}_j$ (estimated Koopman mode), ${\hat P}_{j}^{k}({\bf x}^0)$ (estimated PF)
\end{algorithm}
\end{figure}

\section*{Acknowledgment}

The authors appreciate Professor Atsushi Ishigame (Osaka Metropolitan University) for his constant support of the authors' collaboration.




\end{document}